\documentclass[pdflatex,sn-mathphys-num]{sn-jnl}

\usepackage{graphicx}%
\usepackage{multirow}%
\usepackage{amsthm}%
\usepackage{mathrsfs}%
\usepackage[title]{appendix}%
\usepackage{xcolor}%
\usepackage{textcomp}%
\usepackage{manyfoot}%
\usepackage{fullpage}
\usepackage{booktabs}%
\usepackage{algorithm}%
\usepackage{algorithmicx}%
\usepackage{algpseudocode}%
\usepackage{listings}%
\usepackage[dvipsnames]{xcolor}
\usepackage{amsmath,amsfonts,amsthm,bm,mathrsfs, relsize} %

\newtheorem{theorem}{\bf Theorem}[section]

\newtheorem{remark}{\bf Remark}[section]
\newtheorem{definition}{\bf Definition}[section]
\newtheorem{proposition}{\bf Proposition}[section]

\newtheorem{lemma}{\bf Lemma}[section]

\newtheorem{assumption}{\bf {Assumption}}[section]

\newtheorem{example}{\bf Example}[section]
\newtheorem{note}{\bf Note}[section]
\numberwithin{equation}{section}
\newcommand{\q}{\hspace{.1cm}}
\newcommand{\ra}{\rightarrow}

\newcommand{\xt}{\text}
\newcommand{\prob}{\mathbb{P}}
\newcommand{\E}{\mathbb{E}}
\newcommand{\md}{\mathbb{D}}

\newcommand{\ol}{\overline}

\newcommand*\diff{\mathop{}\!\mathrm{d}}

\newcommand{\jt}{\mathbb{J}_1}

\newcommand{\esp}{\mathrm{ess sup}}
\makeatletter
\def\widebreve{\mathpalette\wide@breve}
\def\wide@breve#1#2{\sbox\z@{$#1#2$}%
\mathop{\vbox{\m@th\ialign{##\crcr
\kern0.08em\brevefill#1{0.8\wd\z@}\crcr\noalign{\nointerlineskip}%
$\hss#1#2\hss$\crcr}}}\limits}
\def\brevefill#1#2{$\m@th\sbox\tw@{$#1($}%
\hss\resizebox{#2}{\wd\tw@}{\rotatebox[origin=c]{90}{\upshape(}}\hss$}
\def\widecheck{\mathpalette\wide@breve}
\def\wide@check#1#2{\sbox\z@{$#1#2$}%
	\mathop{\vbox{\m@th\ialign{##\crcr
				\kern0.08em\checkfill#1{0.8\wd\z@}\crcr\noalign{\nointerlineskip}%
				$\hss#1#2\hss$\crcr}}}\limits}
\def\checkfill#1#2{$\m@th\sbox\tw@{$#1($}%
	\hss\resizebox{#2}{\wd\tw@}{\rotatebox[origin=c]{90}{\upshape(}}\hss$}
\numberwithin{equation}{section}  

\begin{document}

\title[Article Title]{Functional limiting behaviour of non-stationary marked Hawkes processes under multi-scaling high-intensity regime}


\author*[1]{\fnm{Ankita} \sur{ Sen}}\email{atikna.math@gmail.com}

\author*[2]{\fnm{Dharmaraja} \sur{   Selvamuthu}}\email{dharmar@maths.iitd.ac.in}

\author*[3]{\fnm{N.} \sur{ Selvaraju}}\email{nselvaraju@iitg.ac.in}
\affil[1]{\orgdiv{Faculty of Data and Decision Sciences}, \orgname{Technion--Israel Institute of Technology}, \orgaddress{\city{Haifa}, \postcode{3200003}, \country{Israel}}}

\affil[2]{\orgdiv{Department of Mathematics}, \orgname{Indian Institute of Technology Delhi}, \orgaddress{\city{New Delhi}, \postcode{110016}, \country{India}}}

\affil[3]{\orgdiv{Department of Mathematics}, \orgname{Indian Institute of Technology Guwahati}, \orgaddress{\city{Guwahati}, \postcode{ 781039}, \country{India}}}



\abstract{
This paper studies the non-stationary marked Hawkes processes in which  the base intensity function  is time-dependent and the kernel function is governed by an external random factor called ``mark". The marks  are assumed to be determined by random events occurrence over time resulting in a non-identical distribution. These facts introduce two kinds of difficulties: one from the absence of the i.i.d. (independent and identical distributed) and the other is from the non-stationarity. In this framework, an asymptotic regime, often referred to as the \textit{high intensity regime}, is considered, under which the intensity increases with time. This can be obtained by multiplying the time parameter by $n^{\alpha}$, for some $\alpha>0$. Under the high intensity regime, the functional law of large numbers and the functional central limit theorem are established. 
The rescaled (centered and normalized) marked Hawkes process is proved to converge in distribution to a Gaussian process, which is the accumulation of Gaussian noise and a diffusion process. In addition, a shot-noise process is studied,  
and the functional Limit Theorems under appropriate hypotheses are established.
}




\maketitle

\section{Introduction}\label{section_introduction}
A \textit{point process} is a well-known mathematical model in which a large section of research on point processes has been developed on the \textit{Poisson process} and its natural extension, the non-homogeneous Poisson process, which are suitable for scenarios where the inter-occurrence times are independent. However, many practical phenomena exhibit correlated inter-occurrence times or burstiness in the associated point processes. One way to capture such burstiness or auto-correlation is to incorporate random features within the intensity processes. On this account, a special class of point process, called the \textit{Markov-modulated Poisson processes} (or \textit{Markov-modulated non-homogeneous Poisson processes} (\cite{sen2023diffusion})), is widely applicable in various mathematical models, such as telecommunications systems, queueing systems, mathematical biology, weather modeling and  data traffic areas due to its ability to capture dynamic behaviour. Nevertheless, in many real-life phenomena, the occurrence of one event makes the occurrence of the next event more likely to follow in quick succession, resulting in clusters in the sequence of events of arrivals. As a result, this type of phenomenon can be better modeled using variables that are not memory-less, allowing occurrences to have an impact on subsequent events and resulting in an increment of dependence. A very special class of point processes introduced by \cite{hawkes1971point} is a suitable example of  point processes, where the intensity is determined by the entire history, effectively capturing auto-correlation, over-dispersion, and clustering effects in the counting processes. In the literature, these self-exciting processes are named after Hawkes and are referred to as the \textit{Hawkes processes}, which have wide-ranging applications in finance (\cite{bacry2015hawkes}, \cite{Dharmaraja_book}, \cite{hawkes2018hawkes}) and queueing systems (\cite{selvamuthu2022infinite}, \cite{gao2018functional}).

Hawkes processes are characterized through the intensity process formulation, which is  derived by a base intensity function and a  kernel with respect to the elapsed time of the events (see \cite{hawkes1971point}, \cite{hawkes1971spectra}). Since the dynamic structure of the intensity process is determined by the entire history of the past random events, the process is called {\it self-exciting}. Consequently, the higher the intensity, the higher the possibility of occurrence. However, in many studies, the kernel or the self-exciting function in the formulation of the intensity process may be governed by an exogenous random process, characterized as ``marks'' corresponding to each random event (see \cite{bremaud2002power}, \cite{bremaud2002rate}, \cite{horst2021functional}, and the references therein). This point process is called \textit{marked Hawkes process}. In these aforementioned studies, the marks are assumed to be independent and identically distributed (i.i.d.). However, the statistical evidence presented in \cite{ogata1988statistical} suggests that the marks might be correlated or that their dynamics might be influenced by the evolution of the associated random times. The focus of this paper is to study the marked Hawkes process in which the marks sequence does not follow the ideal i.i.d. property, since their dynamics are highly influenced by the associated random occurrence times, and their mark distribution is governed by the evolution of these random occurrence times. Recently, a study by \cite{li2022functional} on the marked Hawkes process considered a partly similar problem and characterizes it through the immigration–birth (branching) representation.

We consider a marked Hawkes process $\{N^{mH}(t), t\geq 0\}$ with intensity process $\{\lambda^{mH}(t), t\geq 0\}$ defined by 
\begin{equation}\label{def_intensity}
\lambda^{mH} (t):= \mu^{mH} (t)+\sum_{i=1}^{N^{mH}(t)} \phi(t-\tau_{i}, \eta _{i}(\tau_{i})) , \q t\geq 0,
\end{equation} 
where $\{\tau_{i}, i\in \mathbb{N} \}$ is a sequence of random occurrence times, and $\{\eta_{i}(\tau_{i}), i\in \mathbb{N}\}$ is the corresponding mark sequence determined by the time epoch $\tau_{i}$ taking the values from a measurable space $(\mathcal E, \mathscr{B}_{\mathcal E})$. For every $i\in \mathbb{N}$, the conditional distribution of $\eta_{i}(\tau_{i})$ is given by
\begin{equation}\label{def_dist_function_Z}
\prob \left( \eta_{i}(\tau_{i}) \in \mathcal{A}|\tau_{i}=s, \tau_{j}, j\leq i  \right) =\pi_{s}(\mathcal A), \q \mathcal A\in \mathscr{B}_{\mathscr{E}}, \q s\in \mathbb{R}_{+}.
\end{equation}
The deterministic function $\mu^{mH}(.) : \mathbb{R}_{+}\ra \mathbb{R}_{+}$ in the representation (\ref{def_intensity}) is often referred to as the base intensity function and the kernel $\phi(.,.): \mathbb{R}\times \mathcal E \ra \mathbb{R}_{+}$ as the self-exciting function. 
The non-stationary  behaviour appears in the formulation of the marked Hawkes process due to its non-stationary base intensity function $\mu^{mH}(.)$ and the time-dependent distribution function $\pi_{t}(.)$ for the mark sequence $\{\eta_{i}(.), i\in \mathbb{N}\}$. By the construction of (\ref{def_intensity}), we observe that the marked Hawkes process $\{N^{mH}(t), t\geq 0\}$ (even the Hawkes process) does not have the Markov property. But, in particular, if the kernel assumes an exponential form, i.e., $\mu^{mH}(.)=a+e^{-\gamma t}(\mu^{mH}_{0}-a)$ and $\phi(t,z)=be^{-\gamma t} $ for some $a>0$, $\mu^{mH}_{0}>0$, and $b\geq 0$, $\gamma \geq 0$, the  process $\{N^{mH}(t), t\geq 0\}$ does not satisfy the Markov property but together with the intensity process $\{\lambda^{mH}(t), t\geq 0\}$, namely,  the two-dimensional random process $\{(N^{mH}(t), \lambda ^{mH}(t)), t\geq 0\}$ can be proved to satisfy the Markov property.
 Moreover, we point out that we generalize the marked point process study into the higher step, where the mark sequence $\{\eta_{i}(.), i\in \mathbb{N}\}$ no longer follows the i.i.d. property and their distribution is determined by the random occurrence time. Considering the self-exciting function $\phi(.,.)$ under the general framework in the representation (\ref{def_intensity}) (exponential as well as non-exponential), our study covers a significant large domain in the literature of the marked Hawkes process (Markov as well as non-Markov). In addition, excluding the mark dependency in the kernel formation $\phi(.,.)$ in (\ref{def_intensity}), the marked Hawkes process becomes a simple Hawkes process as discussed in Section \ref{section_def_marked Hawkes process}.

We particularly focus on the limiting study of the marked Hawkes process under a certain asymptotic regime, where the intensity becomes larger and larger as the time parameter grows towards infinity. We consider a sequence of marked Hawkes processes $\{N^{mH}_{n}(t), t\geq 0\}$ by indexing the scaling parameter $n\in \mathbb{N}$ along with the intensity process $\{\lambda^{mH}_{n}(t), t\geq 0\}$. We assume that the stochastic evolution of the intensity entity is described by $\lambda ^{mH}_{n}(t)= \lambda^{mH} (n^{\alpha}t)$ for all $t\geq 0$, for some $\alpha>0$, which immediately results the cumulative intensity $\int_0^{n^\alpha t} \lambda ^{mH}(s)  \diff s = \mathcal {O}(n^{\alpha})$ approaching to $\infty$ as $n\ra \infty$.
As a consequence, we refer to this particular asymptotic regime as \textit{high-intensity asymptotic regime}, which is significantly different from the conventional asymptotic regime used in the limit study of the Hawkes process as well marked Hawkes processes. Under the multi-scaling high-intensity regime, we establish the functional law of large numbers and the functional central limit theorem for the scaled marked Hawkes process along with the cumulative intensity process. In particular, to derive the functional law of large numbers, we rescale $N^{mH}_{n}(t)$ by $\frac{N^{mH}_{n}(t)}{n^{\alpha}}$, and accordingly we show in Theorem \ref{thm_LLN_N^{mH}(t)} that the rescaled process converges in distribution to a deterministic process $\{\ol{N}^{mH}(t), t\geq 0\}$ as $n\ra \infty$ for any $\alpha>0$. Here, the resulting dynamics of the limit $\{\ol{N}^{mH}(t), t\geq 0\}$ is determined by the base intensity and the kernel or the self-exciting functions (more appropriately resolvent kernel). Knowing the deterministic convergence from Theorem \ref{thm_LLN_N^{mH}(t)}, we define the diffusion scaled version for marked Hawkes processes, denoted by $\{\widehat{N}^{mH}_{n}(t), t\geq 0\}$. 
We establish the functional central limit theorem in Theorem \ref{thm_FCLT_N^{mH}(t)} under the multi-scaling high-intensity regime by proving the weak convergence of the diffusion scaled process $\{\widehat{N}^{mH}_{n}(t), t\geq 0\}$ to a Gaussian process $\{\widehat{N}^{mH}(t), t\geq 0\}$ as $n\ra \infty$, which is accumulation by the Gaussian noise $\{\mathcal W(t) , t\geq 0\}$ and the diffusion process $\{\widehat{\mathcal{B}}^{mH}(t), t\geq 0\}$. Here the diffusion approximation $\{\widehat{\mathcal{B}}^{mH}(t), t\geq 0\}$ associated with the Gaussian process $\{\mathcal B(t), t\geq 0\}$ (see Proposition \ref{prop_fclt_upsilon_{n}(t)}) results from the functional central limit theorem for cumulative intensity process (see Theorem \ref{thm_FCLT_intensity_process}) inside the high-intensity regime. Recall that, under the multi-scaling asymptotic regime, the range of diffusion scale parameter $\delta$ plays a crucial role in establishing the functional central limit theorem, since, this is entirely determined by the scaling parameter $\alpha$, whereas inside the conventional limiting regime $\delta$ is often considered to be $1/2$.

 The high-intensity regime presented in this article is built on a generalized asymptotic framework applicable to large classes of self-exciting functions. Henceforth, our proposed convergence results are widely applicable to more advanced and general frameworks complementing the existing limiting results. If we restrict ourselves to a special case assuming $\{\eta_{i}(.), i\in \mathbb{N}\}$ i.i.d., our limiting results recover the functional central limit theorem for marked (i.i.d.) Hawkes processes, established in \cite{horst2021functional} (see Remarks \ref{remark_LLN_i.i.d_mark}, and Remark \ref{remark_fclt_i.i.d}). Moreover, if the kernel $\phi(.,.)$ does not depend on the marks, our functional limit theorem can be generalized into multi-dimensional case obtaining the same results presented by \cite{bacry2013some} (see Remarks \ref{remark_LLN_simple Hawkes} and Remark \ref{remark_fclt_simple Hawkes}). Earlier, \cite{gao2018limit} studied the functional central limit theorem for simple Hawkes processes in a Markovian framework, considering an exponential self-exciting function under the assumption that the base (initial) intensity function is of $\mathcal {O}(n)$ and generating the high-intensity regime for large $n$. Recently, \cite{li2022functional} studied the functional central limit theorem for non-stationary marked Hawkes processes under the high intensity asymptotic regime, adapting immigration–birth branching representation. However, the asymptotic regime considered by \cite{li2022functional} is different from our present work, and therefore the limit process is different from our proposed limiting theorems. Finally, we analyze the marked Hawkes process formulated in terms of a multi-scaled   high intensity asymptotic framework.

In addition, we formulate a shot noise process $\{Y^{mH}_{n}(t), t\geq 0\}$ generated from the marked Hawkes measure $N^{mH}_{n}(\diff s, \diff z) $, which is an additive functional process determined by a given kernel function $\theta_{n}(t,z)$ (exponential or non-exponential). Under suitable assumptions, we derive the functional law of large numbers and functional central limit theorem for the shot noise process $\{Y^{mH}_{n}(t), t\geq 0\}$ within the multi-scaling high intensity regime in Section \ref{section_proofs_main_thoerems}. Similar to the limiting processes for the marked Hawkes processes, the functional law of large numbers results in a deterministic process determined by the initial base intensity function and the self-exciting function. While the limiting process resulting from the functional central limit theorem is a sum of the Gaussian noise $\{\mathcal{W}_{\ol{\theta}}(t), t\geq 0\}$ and diffusion process $\{\widehat{\mathcal{B}}^{mH}(t), t\geq 0\}$ associated with the deterministic convergence of the kernel function.

Our main contribution is in establishing the functional limit theorems for the marked Hawkes process along with the intensity process inside the proposed high intensity multi-scaling regime under a non-stationary framework. The main key challenge lies in proving the functional law of large numbers and the functional central limit theorem as the marked sequence in the formulation of marked Hawkes process no longer follows the ideal i.i.d. property under the non-stationary framework. Moreover, the incorporation of time-dependency (non-stationary features) inside the marked sequence, and their distribution, makes the limiting study significantly challenging, taking into account that conventional methodologies to prove the limiting theorems will not hold anymore. The time dependence of both the mark distributions and the excitation mechanism necessitates a refined analysis of the interactions among the arrival dynamics, mark scaling, kernel concentration, and intensity fluctuations.

The exponents $\alpha$ and $\beta$ provide a joint time-mark scaling framework for the marked Hawkes process. The parameter $\alpha (>0)$ determines the temporal acceleration, while $\beta(\geq 0)$ determines the scaling of the marks and the asymptotic region of the mark-dependent kernel. The conventional linear time scaling is recovered by substituting $\alpha=1$. The main contribution is not merely the formal replacement of $n$ by $n^\alpha$ or of $z$ by $n^\beta z$. Rather, the framework allows the event frequency, mark magnitude, kernel concentration, fluid normalization, and diffusion fluctuations to evolve at different polynomial rates under a common asymptotic parameter. 
However, its value alone does not determine which scaling dominates but controls the magnitude of the temporal decay and mark dependence of the kernel.

The methodology we adapted  is relying on the martingale construction of the point process, which is highly inspired by \cite{bremaud1981}. Firstly, we represent the intensity formulation (\ref{def_intensity}) in terms of stochastic integral equation with respect to a Poisson measure defined on an extended probability space, as discussed in Section \ref{section_def_marked Hawkes process}. Under the non-stationary framework, this stochastic integral representation of the intensity process falls in the category of stochastic linear Volterra integral equation with non-convoluted kernel function. Therefore, we need to show the existence of the solution representing the intensity, which is ensured by Becker's form of resolvent (see \cite{burton2005}). The existence of Burton's resolvent kernel is ensured by suitable assumptions on the given kernel function, which provides a tractable representation, which makes the asymptotic analysis comparatively easier. In particular, we adapt the martingale central limit theorem presented in \cite{ethier2009markov} to establish the requisite convergences in the present study. Earlier, \cite{li2022functional} established the functional limit theorem for the marked Hawkes process under a similar framework, which relies on establishing the finite-dimensional convergence and tightness criterion. Nevertheless, \cite{li2022functional} has a relatively complex methodology that differs from that of our current martingale approach. Adapting the proposed methodology has the advantage of being able to establish the limit theorems under stationary frameworks or with i.i.d. mark sequences or by using only simple Hawkes processes without marks.

The rest of this paper is organized as follows. The marked Hawkes process and the marked Hawkes measure  along with the preliminaries are defined in Section \ref{section_def_marked Hawkes process}. In Section \ref{section_main_results_complete} the proposed high-intensity asymptotic regime and the main limit theorems of the present article are introduced. The proofs of the limit theorems are presented in detail in Section \ref{section_proofs_main_thoerems}. The paper is concluded with future work in Section \ref{section_conclusion}.

\subsection{Notation}
Let $(\mathcal  A, d)$ be a metric space with metric $d$, and $\mathscr{B}_{\mathcal A}$ be the Borel sigma algebra on $\mathcal A$. Let $\mathbb{C}\left( \mathcal E, \mathcal  A \right)$ denote the space of all $\mathcal  A$-valued continuous functions defined on $\mathcal E$. Let $\mathbb{C}_{b}\left( \mathcal E, \mathcal  A \right)$ be the subset of bounded functions in $\mathbb{C}\left( \mathcal E,\mathcal  A \right)$, and $\mathbb{C}_{c}\left([0, \infty),\mathcal  A\right)$ be the subspace of $\mathbb{C}\left( \mathcal E,\mathcal  A \right)$ which has compact support. Given a metric space $(\mathcal A, d)$, $\md \left(  \mathcal E, \mathcal  A\right)$ denotes the space of all right-continuous $ \mathcal  A$-valued functions on $\mathcal E$ with left limits (space of all $ \mathcal  A$-valued c\`adl\`ag functions). In literature, the space $\md \left( \mathcal E, \mathcal  A\right)$ is often referred to as \textit{Skorokhod space}. The Skorokhod space $\md \left(\mathcal E, \mathcal  A \right)$ is endowed with the Skorokhod $\jt$ topology, which is induced by the Skorokhod metric.  In particular, when $\mathcal A=\mathbb{R}$ and $\mathcal E=\mathbb{R}$, given $f\in \mathbb{C} \left(\mathcal E, \mathcal  A \right)$, we denote $||f||_{T}=\sup_{0 \leq t \leq T} |f(t)|$, for every $T>0$. Also, if $\mathcal E=\mathbb{R}$, we denote  the norm $||f||_{\mathcal A}=\sup_{x\in \mathcal A}|f(x)|$. For two-parameter function, if $f\in \mathbb{C}(\mathbb{E}\times \mathcal A, \mathbb{E})$, then $||f(.,.)||_{T,\mathcal A}=\sup_{\substack{0\leq t \leq T\\ x\in \mathcal A}}|f(t,x)|$. 
For a sequence $\{X_{n}, n\in \mathbb{N}\}$ of stochastic processes with trajectories in $\md \left(  \mathcal E,\mathcal  A \right)$, 
 we say that $X_{n}$ converges in distribution to $X$ in $\md \left(\mathcal E,\mathcal  A\right)$ embedded with the Skorokhod $\jt$ topology, as $n\ra \infty$. The notion of convergence in distribution is defined in standard way (see \cite{billingsley1999convergence},  \cite{ethier2009markov}).

\section{Marked Hawkes Measure and Marked Hawkes Process}\label{section_def_marked Hawkes process}
Let $(\Omega, \mathscr{F}, \prob)$ be a probability space embedded with the filtration $\mathscr{F}=\{\mathscr{F}(t), t\geq 0\}$. Let $\{\tau_{i}, i\in \mathbb{N} \}$ be a sequence of $\mathscr{F}$-adapted, non-decreasing random times defined on $[0, \infty)$. Define $\{\eta_{i}, i\in \mathbb{N}\}$,
the mark sequence such that 
\begin{equation*}
\prob \left( \eta_{i}(\tau_{i}) \in \mathcal{A}|\tau_{i}=s, \tau_{j}, j\leq i  \right) =\pi_{s}(\mathcal A), \q \mathcal A\in \mathscr{B}_{\mathscr{E}}, \q s\in \mathbb{R}_{+}.
\end{equation*}
Here we assume that the mark $\eta_{i}(.)$  only dependent on the random time $\tau_{i}$ and  that are independent of the remaining time epochs $\tau_{j}$,, for $j<i$.

Let the random sequence $\{\tau_{i}, i\in \mathbb {N}\}$ and $\{\eta_{i}(\tau_{i}), i\in \mathbb{N}\}$ be the $\mathscr{F}$-adapted \textit{marked point measure} $N^{mH}(\diff t, \diff z)$ given by
\begin{equation}\label{def_Hawkes_measure_N}
N^{mH}(\diff t, \diff z):=\sum_{i\in \mathbb{N}} \bm{1}(\tau_{i}\in \diff t, \eta_{i}(\tau_{i}) \in \diff z).
\end{equation}
The marked point measure $N^{mH}(\diff t, \diff z)$ is a transition measure $(\Omega,\mathscr{F} )$ into $(\mathbb{R}_{+}\times \mathcal E, \mathscr{B}_{\mathbb {R}_{+}} \otimes \mathscr{B}_{\mathcal E} )$, which is $\sigma$-finite if and only if the sequence $\{\tau_{i}, i\in \mathbb{N}\}$ is $\mathscr{F}$-non-explosive. 
The point process $\{N^{mH}(t), t\geq 0\}$ embedding the marked point measure $N^{mH}(\diff t, \diff z )$ is said to be \textit{$\mathcal E$-marked point process} if 
\begin{equation*}
N^{mH}(t):=N^{mH}((0,t], \mathcal E), \q t\geq 0.
\end{equation*}

\begin{definition}\label{def_marked_Hawkes_process}
	Let $\lambda^{mH}(.)$ be given by
\begin{equation*}
\lambda^{mH} (t):= \mu^{mH} (t) + \sum_{i=1}^{N^{mH}(t)} \phi\left(t-\tau_{i}, \eta _{i}(\tau_{i})\right) , \quad t\geq 0,
\end{equation*}
	where $\mu^{mH}:\mathbb{R}_{+} \ra \mathbb{R}_{+}$ is a locally integrable deterministic function 
	and the kernel $\phi(.,.): \mathbb{R}_{+} \times \mathcal E \ra \mathbb{R}_{+}$ is a $\mathscr{B}_{\mathbb{R}_{+}} \otimes \mathscr{B}_{\mathcal{E}}$-measurable function. 
 For each $t \geq 0$, if the $\mathcal E$-marked point process $N^{mH}(t)$ admits $\mathscr{F}$-intensity $\lambda ^{mH}(t)$, we say that  
 $\{N^{mH}(t), t\geq 0\}$ is a marked Hawkes process. The function $\mu ^{mH}(.)$ is often referred to as baseline intensity, and the kernel $\phi(.,.)$ is called as the exciting function.  
\end{definition}

\begin{definition}\label{def_marked_Hawkes_measure}
	Let $N^{mH}(\diff t, \diff z)$ be the $\mathcal E$-marked Hawkes measure defined on $(\mathbb{R}_{+}\times  \mathcal E, \mathscr{B}_{\mathbb{R}_{+}} \otimes \mathscr{B}_{\mathcal E})$. Suppose  for each $\mathcal A\in \mathscr{B}_{\mathcal E}$, $\{N^{mH}((0,t], \mathcal A), t\geq 0\}$ admits the $\mathscr{F}$-intensity process $\{\lambda^{mH} (t) \pi_{t}(\mathcal A), t\geq 0\} $, where $\pi_{t}( \mathcal A) $ is a transition measure from $(\Omega \times \mathbb{R}_{+}, \mathscr{F}\otimes \mathscr{B}_{\mathbb{R}_{+}})$ into $(\mathcal E, \mathscr{B}_{\mathcal E})$. Then $N^{mH}(\diff t, \diff z)$ admits the intensity measure $\lambda^{mH} (t-) \pi_{t}(\diff z) \diff t$.
\end{definition}
The intensity process $\{\lambda^{mH} (t), t\geq 0\}$ is assumed to be non-negative, locally integrable and absolutely continuous with respect to the Lebesgue measure. The intensity entity $t\mapsto \lambda^{mH}(t)$ is characterized by the infinitesimal relation 
\begin{equation*}
\prob \left( N^{mH}(t+\diff t)-N^{mH}(t) \neq 0| \mathscr{F}(t)\right) = \lambda^{mH}(t) \diff t
+ o(\diff t).
\end{equation*}

Now, we recall some well known definitions as a remarked of the main Definition \ref{def_marked_Hawkes_measure} as follows.

\medskip{}\noindent\textbf{Point Process:} If the state space $\mathcal{E}$ contains a single point, then the process $\{N^{p}(t), t\geq 0\}$ associated with the point measure 
	\begin{equation*}
	N^{P}((0,t])=\sum_{i\in\mathbb{N}} \bm {1}(\tau_{i} \leq t), \quad  t\geq 0,
	\end{equation*}
	becomes the point process, where $N^{p}(t)=N^{p}((0,t])$, for all $t\geq 0$. 

\medskip{}
\noindent
\textbf{Hawkes Process:} \label{def_Hawkes_point_process}
If the kernel $\phi(t,z)=G(t)$ for some $\mathscr{B}_{\mathbb{R}_{+}}$-measurable function $G:\mathbb{R}_{+}\ra \mathbb{R}_{+}$, for every $z\in \mathcal{E}$, then we characterize the point process $\{N^{H}(t), t\geq 0\}$ as a Hawkes process (instead of marked Hawkes process) with intensity process $\{\lambda^{H}(t), t\geq 0\}$ with 
	\begin{equation*}
	\lambda^{H}(t)=\mu^{H}(t)+ \sum_{i=1}^{N^{H}(t)} G(t-\tau_{i}), \quad t\geq 0.
	\end{equation*}

\medskip{}
\noindent
\textbf{Multivariate Hawkes Process:} If the $\mathcal{E}=\{z_1, z_2, \ldots, z_{K}\}$, then we define
	\begin{equation*}
	N_{j}(t)=\sum_{i\in \mathbb{N}} \bm {1}(\tau_{i}\leq t, \eta_{i} =z_{j}), \q t\geq 0,\q j=1,2,\ldots K,
	\end{equation*}
	and the process $\{(N_1(t), N_2(t), \ldots, N_{K}(t)), t\geq 0\}$ is a  \textit{multi-variate Hawkes process} with intensity process 
	\begin{equation*}
	\lambda_{j}(t)=\mu_{j}(t)+ \sum_{i=1}^{N_{j}(t)} H_{ij}(t-\tau_{i}), \q j=1,2, \ldots, K,
	\end{equation*}
where $\phi(t - \tau_i, z_j)= H_{ij}(t-\tau_i)$.

\subsection{Martingale Representation of Marked Hawkes Process}\label{section_martingale_marked_Hawkes}
Let  $\mathscr{F}^{mH}(t):=\sigma (N^{mH}((0,s], \mathcal A), s\leq t, \mathcal A\in \mathscr{B}_{\mathcal E})$, and let $\{\mathscr{F}^{mH}(t), t\geq 0\}$ be the associated right-continuous filtration. Let $\{\mathscr{F}(t), t\geq 0\}$ be another filtration such that $\mathscr{F}^{mH}(t)\subset \mathscr{F}(t)$, for each $t\geq0$.

Given a $\mathscr{B}_{\mathbb{R}_{+}} \otimes \mathscr{B}_{\mathcal{E}}$-measurable $\mathscr{F}$-predictable $\mathscr{E}$-marked process $\{C(t,z), t\geq 0, z\in \mathcal E\}$, by the stochastic integral representation we get that
\begin{equation}\label{def_int_reprep_N^{mH}}
\int_{0}^{\infty} \int_{\mathcal E} C(s,z) N^{mH}(\diff s, \diff z)= \sum_{i\in \mathbb{N}} C(\tau_{i}, \eta_{i}(\tau_{i})) \bm{1}(\tau_{i} \leq \infty).
\end{equation}
Due to conditional independence property of the marked sequence $\{\eta_{i}(\tau_{i}), i\in \mathbb{N}\}$ by the time sequence $\{\tau_{i}, i\in \mathbb{N}\}$, the intensity measure of the marked Hawkes  measure $N^{mH}(\diff t, \diff z)$, (\ref{def_Hawkes_measure_N}), is given by 
$\lambda^{mH}(t-) \pi_{t}(\diff z) \diff t$, which is equivalent to say that for every $\mathscr{F}$-predictable $\mathcal E$-marked process $C$,
\begin{equation*}
\E \left(\int_0^t \int_{\mathcal {E}} C(s,z) N^{mH}(\diff s, \diff z) \right) = 
\E \left(\int_0^t \int_{\mathcal{E}} C(s,z) \lambda^{mH} (s)  \pi_{s}(\diff z) \diff s\right) .
\end{equation*}
Therefore, the integral representation with respect to marked Hawkes measure $N^{mH}(\diff t, \diff z)$ given by (\ref{def_int_reprep_N^{mH}}) leads to the integral representations 
\begin{equation}\label{def_Hawkes_int_rep}
\lambda^{mH} (t)=\mu^{mH} (t)+ \int_{0}^{t} \int_{\mathcal{E}} 
\phi(t-s, z)  N^{mH}(\diff s, \diff z),
\end{equation}
and its compensated measure is given by
\begin{equation}\label{def_M^{mH} measure}
 M^{mH}(\diff t, \diff z):=N^{mH}(\diff t, \diff z)-\lambda^{mH} (t)\pi_{t}(\diff z)  \diff t.
\end{equation}
The integral process with respect to the compensated marked measure defined by 
\begin{equation*}
\mathcal I^{mH}_{C} (t):=\int_0^t \int_{\mathcal {E}} C(s,z) M^{mH}(\diff s, \diff z),  \q t\geq 0,
\end{equation*}
and the next lemma, associated with  $\mathcal I^{mH}_{C}(t)$ yields the useful martingale characterization in virtue of \cite{bremaud1981}. 
\begin{lemma}\label{lemma_N^{mH}_martingale}
	Let $N^{mH}(\diff t, \diff z)$ be a marked Hawkes measure with $\mathscr{F}$-intensity measure $\lambda^{mH} (t-)\pi_{t}(\diff z)  \diff t$. Given a $\mathscr{B}_{\mathbb{R}_{+}} \otimes \mathscr{B}_{\mathcal{E}}$-measurable $\mathscr{F}$-predictable $\mathcal E$-marked process $\{C(t,z), t\geq 0, z\in \mathcal E\}$, if
	\begin{equation}
	\int_0^t \int_{\mathcal{E}}	|C(s,z)|\lambda^{mH} (s) \pi_{s}(\diff z)  \diff s <\infty, \q \prob-a.s.,
	\end{equation}
	then the integral process $\{\mathcal I^{mH}_{C}(t), t\geq 0\}$ is $\mathscr{F}$-local martingale.
\end{lemma}

In order to represent the marked Hawkes measure with respect to an independent Poisson measure, we are going to follow a procedure inspired by  Theorem 7.4 of \cite{ikeda1989stochastic}.

Given the probability space $(\Omega, \mathscr{F}, \mathbb{P})$, there is an extended probability space $(\widetilde{\Omega}, \widetilde{\mathscr{F}}, \widetilde{\prob})$, and a Poisson measure $P(\diff s, \diff z, \diff u)$ defined on $\mathbb{R}_{+} \times \mathcal{E}\times \mathbb{R}_{+}$ with intensity measure $\diff s \ \pi_{s}(\diff z)\diff u$ such that 
\begin{equation*}
N^{mH}((0,t], \mathcal A)=  \int_0^t \int_{\mathcal{A}}  \int_{\mathbb{R}_{+}}
\bm{1}(u\leq \lambda^{mH}(s-))  P(\diff s, \diff z, \diff u), \q \mathcal A\in \mathscr{B}_{\mathcal E}.
\end{equation*}
Therefore, the representation of (\ref{def_Hawkes_int_rep}) leads to 
\begin{equation}\label{rep_lambda_w.r.t_Poisson}
\lambda^{mH}(t)= \mu^{mH}(t)+ \int_0^{t} \int_{\mathcal E} \int_{0}^{\lambda^{mH}(s-)}
\phi(t-s, z) P(\diff s, \diff z, \diff u).
\end{equation}
Next, our aim is to derive a martingale representation of the intensity process $\{\lambda ^{mH}(t), t\geq 0\}$ with respect to compensated measure 
\begin{equation}\label{def_M^{P}(t,z,u)}
M^{P}(\diff t, \diff z, \diff u)=P(\diff t, \diff z, \diff u)-\diff t \ \pi_{t}(\diff z)\diff u.
\end{equation}
Given the kernel $\phi(t,z)$, we introduce the function $\psi(t,s)$ weighted by the marked distribution $\pi_{t}(\diff z)$ defined by 
\begin{equation}\label{def_psi(t,s)}
\psi (t,s):= \int_{\mathcal E} \phi(t-s,z) \pi_{s}(\diff z) .
\end{equation}
Therefore, taking expectation in both side of (\ref{rep_lambda_w.r.t_Poisson}) we obtain 
\begin{equation}\label{def_E(lambda^{mH}(t))}
\E (\lambda^{mH}(t))= \mu^{mH}(t)+ \int_{0}^{t} \psi(t,s) \E (\lambda^{mH}(s)) \diff s,
\end{equation}
which turns out to be the Volterra integral equation of second kind with kernel function $\psi(t,s)$. Since under the hypotheses, the \textit{weighted kernel} $\psi(t,s)$ has non-convolutional structure, the solution of the Volterra integral equation is ensured by Becker's form of resolvent (see Section 7, \cite{burton2005}). On that account, given the weighted kernel function $\psi(t,s)$, we define the \textit{resolvent} equation by virtue of  \cite{burton2005} 
\begin{equation}\label{def_Gamma(t,s)}
\Gamma (t,s)= \psi(t,s) + \int_s^t \psi(t,u)\Gamma (u,s) \diff u, \quad 0\leq s \leq t.
\end{equation}
Using the above analogy, the solution of (\ref{def_E(lambda^{mH}(t))}) can be written as 
\begin{equation}\label{def_expected_intensity_w.r.t gamma}
\E (\lambda ^{mH}(t))= \mu^{mH}(t)+ \int_0^t \Gamma (t,s) \mu^{mH}(s) \diff s,
\end{equation}
which represents the expected valued of intensity $\lambda^{mH}(t)$. Here the resolvent kernel $\Gamma(t,s)$ is defined as follows:
\begin{equation*}
\Gamma (t,s):= \sum_{k=1}^{\infty} \psi^{k}(t,s), \quad \psi^{1}(t,s) := \psi(t,s),
\end{equation*}
where 
\begin{equation*}
 \psi^{k+1}(t,s):=  \int_{s}^{t} \psi(t,v) \psi^{k}(v,s) \diff v, \q \q k\geq 1.
\end{equation*}
The definition of the resolvent kernel $\Gamma(t,s)$ itself yields that 
\begin{equation}\label{psi_subcritical_condition}
    \sup_{s\geq 0} \int_{s}^{\infty} |\psi(t,s)| \diff t < 1,
\end{equation}
which is often referred to be \textit{subcritical condition}. The condition (\ref{psi_subcritical_condition}) essentially maintains the stability for the marked Hawkes measure $N^{mH}(\diff s , \diff z)$ in the long time period. It essentially implies the stability of the resolvent kernel $\Gamma(t,s)$ given by
\begin{equation}\label{gamma_subcritical_condition_first}
    \sup_{s\geq 0} \int_{s}^{\infty} |\Gamma(t,s)| \diff t < \infty.
\end{equation}

\begin{remark}\label{remark_iid of mark sequence}
	In particular, if $\{\eta_{i}(\tau_{i}), i\in \mathbb{N}\}$ is a sequence of i.id marks, independent by the random time epochs $\{\tau_{i}, i\in \mathbb{N}\}$,
	\begin{equation*}
	\prob (\eta_{i}({\tau_{i}}) \in \mathcal A|\tau_{j}, j\in \mathbb{N} )=\pi(\mathcal A), \q i\in \mathbb{N},\q \mathcal A \in \mathscr{B}_{\mathcal E},
	\end{equation*}
    then we have $\psi(t,s)=\psi(t-s)$, for all $0\leq s \leq t$. It follows from (\ref{def_E(lambda^{mH}(t))}) that
    \begin{equation*}
    \E (\lambda^{mH}(t))= \mu^{mH}(t)+ \psi \star \E (\lambda^{mH})(t),
    \end{equation*}
    and hence the solution is given by 
	\begin{equation*}
     \E (\lambda ^{mH}(t))= \mu^{mH}(t)+ \mu^{mH}\star \Gamma (t)
	\end{equation*}
	where the resolvent kernel $\Gamma(.)$ is characterized by the resolvent equation
	\begin{equation*}
	\Gamma (t)= \psi(t) + \psi \star \Gamma(t)=\psi(t)+\Gamma \star \psi(t), \q t \geq 0.
	\end{equation*}
	Therefore, under i.i.d. assumptions, it leads to 
	\begin{equation*}
	||\Gamma||_{\mathbb{L}^{1}}=||\psi||_{\mathbb{L}^{1}}(1+||\Gamma||_{\mathbb{L}^{1}}) ,
	\end{equation*}
	which implies 
	\begin{equation*}
	||\Gamma||_{\mathbb{L}^{1}}=\frac{||\psi||_{\mathbb{L}^{1}}}{1-||\psi||_{\mathbb{L}^{1}}},
	\end{equation*}
	provided $ ||\psi||_{\mathbb{L}^{1}} <1$.
	\end{remark}

In order to derive the limit theorems for the marked Hawkes process $\{N^{mH}(t), t\geq 0\}$ with $\mathscr{F}$-intensity process $\{\lambda ^{mH}(t), t\geq 0\}$, we require the boundedness of the expected entity $\E (\lambda ^{mH}(t))$, $t\geq 0$ stated below in Lemma \ref{lemma_bound of E lambda^{mH}(t)} , which is ensured by the following assumptions.
\begin{assumption}\label{assm_initial_intensity_bound}
	The base intensity function $\mu^{mH}(.) \in \mathbb{L}^{\infty}_{loc}(\mathbb{R}_{+})$, the space of all locally-bounded integrable function on $\mathbb{R}_{+}$.
\end{assumption}
\begin{assumption}\label{assm_kernel_locally_bounded}    
    The kernel function $\phi(t,z) $ is locally-integrable jointly on $(t,z)$ defined on $\mathbb{R}_{+}\times \mathcal E$,
    \begin{equation*}
        \int_{0}^{T}\int_{\mathcal{E}}  |\phi(t,z) | \pi_{t}(\diff z) \diff t <\infty, \; \text{for any} \; T>0,
    \end{equation*}
    i.e. $\phi(.,.)\in \mathbb{L}^{1}_{loc}(\mathbb{R}_{+} \times \mathcal{E}; \diff t \times \pi_{t} (\diff z))$.

    Moreover, we assume that $\phi(.,.)$ satisfies the stronger condition 
		\begin{equation}\label{assm_phi_integral_exist}
		\esp_{z\in \mathcal E} \int_0^{\infty} 
        \left | \phi(t,z) \right| \diff t  <\infty,
		\end{equation}
    i.e., the improper integral (\ref{assm_phi_integral_exist}) exists a.e. for $z\in \mathcal{E}$, and the improper integral is defined by 
    \begin{equation}\label{def_Phi_z_limit_phi}
        \bar{\Phi}(z):=\int_0^{\infty} \phi(t,z) \diff t := \lim_{r\ra \infty} \int_{0}^{r} \phi(t,z) \diff t,
    \end{equation}
    and it is essentially uniformly bounded on $z\in \mathcal{E}$.

    In addition, assume that $\phi(.,z)$ satisfies the uniform tail condition, i.e,  
    \begin{equation}\label{assm_phi_tail_cond}
        \lim_{r\ra \infty} \esp_{z\in \mathcal{E}} \int_{r}^{\infty} |\phi(t,z)| \diff t =0.   
    \end{equation}  
    
\end{assumption}
Here, (\ref{assm_phi_integral_exist}) of Assumption \ref{assm_kernel_locally_bounded} evidently implies that $\psi(.,s) $ and $\Gamma(.,s)$ belong to the class $\mathbb{L}^{1}_{loc}([v, \infty))$ for every $v\geq 0$, and hence the stability conditions (\ref{psi_subcritical_condition}) and (\ref{gamma_subcritical_condition_first}). As the results of the subcritical condition (\ref{psi_subcritical_condition}) and Assumption \ref{assm_kernel_locally_bounded}, one can consider a stronger but simpler sufficient condition 
\begin{equation*}
    \esp_{z\in \mathcal E} \int_0^{\infty} 
        \left | \phi(t,z) \right| \diff t <1.
\end{equation*}

\begin{example}\label{example_phi_nonscaled}
Note that (\ref{assm_phi_integral_exist}) of Assumption  \ref{assm_kernel_locally_bounded} is considered to be a strong condition and it plays a crucial role in establishing that the expected intensity measure is finite. This assumption is substantially different from the assumptions considered in \cite{horst2021functional} due to the ideal i.i.d. feature within the mark sequence $\{\eta_{i}(.), i\in \mathbb{N} \}$.
In many practical settings, the regularity conditions (\ref{assm_phi_integral_exist}), \ref{assm_phi_convergence_z_scaling}) and (\ref{assm_phi_tail_cond}) of Assumption \ref{assm_kernel_locally_bounded} are satisfied by a large class of kernel functions $\phi(.,.)$. We now provide some examples of such $\phi(.,.)$ satisfying Assumption \ref{assm_kernel_locally_bounded}. 
\begin{enumerate}
    \item 
    The dominated class of functions $g\in \mathbb{L}^{1}(\mathbb{R}_{+})$ satisfying $|\phi(t,z)| \leq |g(t)|$, a.e. for $t\geq 0$ and $z\in \mathcal{E}$, e.g., 
    \begin{align*}
        \phi(t,z)= h(z) \exp(-\alpha t) , \; \alpha>0.
    \end{align*}
    with $h\in \mathbb{L}^{\infty}(\mathcal{E})$. 
    \item 
    The separable kernels $\phi(t,z)=h(z) b(t)$ with $ h \in \mathbb{L}^{\infty}(\mathcal{E})$ and $b\in \mathbb{L}^{1}(\mathbb{R}_{+})$, e.g., 
    \begin{align*}
        \phi(t,z)= & h(z)\frac{1}{(1+t)^{\alpha}}, \; \alpha>1,\\
        \phi(t,z)=& h(z) t^{-\beta} \exp(-\alpha t)\bm{1}_{\{t>0\}}, \; 0<\beta<1, \alpha>0.
    \end{align*} 
    \item 
     $\phi(.,.)$ with discontinuous jumps. e.g.,
    \begin{equation*}
        \phi(t,z)= h(z) \bm{1}_{\{0\leq t \leq L\}} , \; L>0.
    \end{equation*}
    \item 
    One can assume that the decay rate function is determined by the mark variable $z$, e.g., 
    \begin{equation*}
        \phi(t,z)= h(z) \exp(-\alpha(z)t), \; \alpha (z)\geq \alpha_0>0, \; \text{a.e}.
    \end{equation*}
\end{enumerate}
\end{example}

Next, we show that the expected intensity measure $ \E \left(\lambda^{mH}(t)\right)$ is finite uniformly on every compact sets, which essentially provides the stability of the intensity estimation over finite time horizon.
\begin{lemma}\label{lemma_bound of E lambda^{mH}(t)}
	Under Assumptions \ref{assm_initial_intensity_bound} and \ref{assm_kernel_locally_bounded}, the expected value of the intensity given by $\E \left(\lambda ^{mH}(.)\right)$ is locally uniformly bounded, i.e., for all $T>0$, 
    \begin{equation*}
        \sup_{0\leq t \leq T} \E \left(\lambda^{mH}(t)\right)  <\infty.
    \end{equation*}
\end{lemma}

Note that the compensated measure $M^{P}(\diff t, \diff z, \diff u)$ defined by (\ref{def_M^{P}(t,z,u)}) is a martingale measure by virtue of Theorem 1.6.3 of Liptser and Shiryaev. Using the definition of $M^{P}(\diff t, \diff z, \diff u)$ by (\ref{def_M^{P}(t,z,u)}), the representation (\ref{rep_lambda_w.r.t_Poisson}) of $\lambda^{mH}(t)$ yields 
\begin{align*}
\lambda ^{mH}(t)=&\mu^{mH}(t)+ \int_0^t \int_{\mathcal E} \int_0^{\lambda^{mH}(s-)} \phi(t-s,z) M^{P}(\diff s, \diff z, \diff u)\\ & + \int_0^t \int_{\mathcal E} \int_0^{\lambda^{mH}(s-)} \phi(t-s,z)  \diff s \pi_{s}(\diff z)\diff u\\
=&\mu^{mH}(t)+ \int_0^t \int_{\mathcal E} \int_0^{\lambda^{mH}(s-)} \phi(t-s,z) M^{P}(\diff s, \diff z, \diff u)+ \int_0^t  \psi(t,s) \lambda^{mH}(s) \diff s. 
\end{align*}
Applying the similar analogy with respect to resolvent kernel $\Gamma(t,s)$ by (\ref{def_Gamma(t,s)}), we obtain 
\begin{align*}
\lambda^{mH}(t)=&\mu^{mH}(t)+ \int_0^t \Gamma (t,s) \mu^{mH}(s) \diff s 
+\int_0^{t} \int_{\mathcal{E}} \int_{0}^{\lambda^{mH}(s-)}
\phi(t-s, z) M^{P}(\diff s, \diff z, \diff u)\\
&\qquad \qquad + \int_0^t \Gamma (t,s) \int_0^s \int_{\mathcal{E}} \int_{0}^{\lambda^{mH}(v-)} \phi (s-v, z)
M^{P}(\diff v, \diff z, \diff u)  \diff s.
\end{align*}
By Fubini's Theorem, we rewrite the second and third terms and we have that
\begin{align}\label{rep_lambda_martingale}
\lambda^{mH}(t)=&\mu^{mH}(t)+ \int_0^t \Gamma (t,s) \mu^{mH}(s) \diff s +
 \int_0^t \int_{\mathcal E} \int_0^{\lambda ^{mH}(v-)} \phi(t-v,z) M^{P}(\diff v, \diff z, \diff u)  \notag \\
 &+ \int_0^t \int_{\mathcal E}\int_0^{\lambda^{mH}(v-)} M^{P}(\diff v, \diff z, \diff u) \int_{v}^{t} \Gamma(t,s) \phi(s-v,z) \diff s.
\end{align}  
Define 
\begin{equation*}
\mathcal{L}_{t}(v,z)=   \phi(t-v,z)+ \int_v^t \Gamma(t,s) \phi(s-v,z)\diff s , \quad t\geq 0.
\end{equation*}

\begin{proposition}\label{prop_intensity_martingale}
	Under Assumptions (\ref{assm_initial_intensity_bound}) and (\ref{assm_kernel_locally_bounded}), the intensity process $\{\lambda ^{mH}(t), t \geq0\}$ is the unique solution of the linear stochastic Volterra equation
	\begin{equation}\label{def_lambda^{mH}(t)_w.r.t martingale}
	\lambda^{mH}(t)=\mu^{mH}(t)+ \int_0^t \Gamma (t,s) \mu^{mH}(s) \diff s +
	\int_0^t \int_{\mathcal E} \int_0^{\lambda ^{mH}(s-)} \mathcal L_{t}(s,z) M^{P}(\diff s, \diff z, \diff u)
	\end{equation}
	where the compensated martingale measure $M^{P}(\diff s, \diff z, \diff u)$ is given by (\ref{def_M^{P}(t,z,u)}). 
\end{proposition}

\noindent 
Given the kernel $\phi(.,.)$ on $\mathbb{R}_{+}\times \mathcal E$ and $\Gamma(.,.)$ on $\mathbb{R}_{+}\times \mathbb{R}_{+}$, 
we introduce
\begin{equation}\label{def_tau_t(v,z)_function}
\mathcal T_{t}(v,z)= \int_v^{t} \phi(w-v,z) \bigg( 1 +\int_{w}^{t}\Gamma (s,w)  \diff s
\bigg) \diff w  .
\end{equation}

\subsection{Shot noise process:}
   In addition to the counting dynamics of the marked Hawkes process, it is often useful to study cumulative functionals generated by its past events. Motivated by this, we define the additive functional process $\{Y^{mH}(t), t\geq 0\}$ in $\md ([0, \infty), \mathbb{R})$ driven by the marked Hawkes measure $N^{mH}(\diff s, \diff z)$, where 
   \begin{equation}\label{def_Y^{mH}}
   Y^{mH}(t)=\sum_{i=1}^{N^{mH}(t)} \theta (t-\tau_{i}, \eta_{i}(\tau_{i})), \q t\geq 0,
   \end{equation}
   which leads to the following integral representation
    \begin{equation}\label{def_Y^{mH}_int}
    Y^{mH}(t)=\int_0^t \int_{\mathcal{E}} \theta \left( t-s,z\right) N^{mH}(\diff s, \diff z).
   \end{equation} 
   Often, in the literature, we refer the additive process $\{Y^{mH}(t), t\geq 0\}$ in $\md ([0, \infty), \mathbb{R})$ as the \textit{shot noise process}. The random dynamics of $Y^{mH}$ determined by $N^{mH}$ provides the insight into the the random evolution of the intensity behaviour in $\md ([0, \infty), \mathbb{R})$.

\section{Multi-scaling High-Intensity Regime and Limit Theorems}\label{section_main_results_complete}
This section studies the functional limit theorems for marked Hawkes processes under a suitable asymptotic regime, which includes functional law of large numbers and functional central limit theorem. In the beginning, we define the scaled marked Hawkes process along scaled the intensity process under multi-scaling high-intensity asymptotic regime framework introducing a pair of scaling parameters $\alpha$ and $\beta$ for some $\alpha>0$ and $\beta \geq 0$. In addition, we define the additive functional process associated with the marked Hawkes measure. Next, we state the main results of the present article in the subsequent Sections \ref{section_FLLN_theorems} and \ref{section_FCLT_theorems}. In particular, the deterministic approximations in virtue of functional law of large numbers of the scaled processes is stated in Section \ref{section_FLLN_theorems}. In the following Section \ref{section_FCLT_theorems}, we discuss the weak convergences of the centered and normalized scaled processes in Skorokhod Space.

Consider a sequence of a marked Hawkes point measure $N^{mH}_{n}(\diff t, \mathcal{A})$ indexed by the subscript $n\in \mathbb{N}$, determined by the random sequence $\{\tau_{i,n}, i\in \mathbb{N}\}$ and $\eta_{i,n}(\tau_{i,n}), i\in \mathbb{N}$, defined in view of (\ref{def_Hawkes_measure_N}) as 
\begin{equation*}
N^{mH}_{n}((0,t],\mathcal{A}):=\sum_{i\in \mathbb{N}}\bm {1}(\tau_{i,n} \leq t, \eta_{i,n}(\tau_{i,n}) \in \mathcal{A}), \q \mathcal{A} \in \mathscr{B}_{\mathcal {E}},
\end{equation*}
where the distribution of the marked sequence is given by 
\begin{equation*}
\prob (\eta_{i,n} (\tau_{i,n})\in \mathcal A|\tau_{i,n}=t, \tau_{j,n}, j\leq i)=\pi_{t}^{n} (\mathcal A ) , \q \mathcal{A}\in \mathscr{B}_{\mathcal E}.
\end{equation*}

Given $\mathcal A\in \mathscr{B}_{\mathcal E}$, we define the scaled marked Hawkes process $\{N^{mH}_{n}((0,t], \mathcal {A}), t\geq 0\}$ as
\begin{equation}\label{def_N^{mH}_scaled}
N^{mH}_{n}((0,t],\mathcal{A}):= N^{mH}((0,n^{\alpha}t], n^{\beta} \mathcal{A} ).
\end{equation}
and the associated  marked Hawkes measure $N_{n}^{mH}(\diff t , \diff z)$ is defined by
\begin{equation}\label{def_N^{mH}_measure scaled}
N^{mH}_{n}(\diff s, \diff z):=N^{mH} (n^{\alpha} \diff s, n^{\beta} \diff z), \q \xt{for some} \q \alpha> 0,\q  \beta \geq 0.
\end{equation}
Define the sigma algebra $\mathscr{F}_{n}^{mH}(t):=\sigma (N^{mH}_{n}((0,s], \mathcal A), s\leq t, \mathcal A\in \mathscr{B}_{\mathcal E})$ and the associated right-continuous filtration $\mathscr{F}_{n}^{mH}:=\{\mathscr{F}_{n}^{mH}(t), t\geq 0\}$. Consider the bigger sigma algebra $\mathscr{F}_{n}^{mH}(t) \subset \mathscr{F}_{n}(t)$, $t\geq 0$ and the filtration $\mathscr{F}_{n}:=\{\mathscr{F}_{n}(t), t\geq 0\}$ such that $\{\lambda^{mH}_{n}(t), t\geq 0\}$ is $\mathscr{F}_{n}$-adapted. 
The definitions (\ref{def_N^{mH}_scaled}) and (\ref{def_N^{mH}_measure scaled}) itself yield that the compensated measure $N^{mH}_{n}(\diff s, \diff z) -\lambda^{mH}(n^{\alpha}s-) n^{\alpha}\diff s \pi_{n^{\alpha}s}(n^{\beta} \diff z) $ is $\mathscr{F}_{n}$- martingale measure and $\lambda^{mH}(n^{\alpha} s-) n^{\alpha}\diff s \pi_{n^{\alpha}s}(n^{\beta} \diff z) $ is the associated $\mathscr{F}_{n}$-predictable intensity measure. Therefore, the scaling of the intensity process and the marked distribution are defined as follows:
\begin{equation}\label{def_lambda_pi_scaled_versions}
    \lambda^{mH}_{n}(t):= \lambda^{mH}(n^{\alpha}t), \quad 
    \pi_{t}^{n}(\mathcal {A}):= \pi_{n^{\alpha}t}(n^{\beta}\mathcal{A} ), \quad t\geq 0, \quad \mathcal{A}\in \mathscr{B}_{\mathcal{E}}.
\end{equation}

For the mark space $\mathcal{E}$, we assume that $\mathcal{E}$ is invariant under the scaling, i.e., 
\begin{equation}\label{invariant_property_mark_space}
    n^{\beta}\mathcal{E}= \mathcal{E}.
\end{equation}
The invariance property of (\ref{invariant_property_mark_space}) guarantees that the scaled subset $n^{\beta}\mathcal{A}$ of $\mathcal{E}$ still gives a subset of the mark space $\mathcal{E}$. It also ensures that $\pi_{s}^{n}(\diff z)$ remains a probability measure on $\mathcal{E}$, as
\begin{equation*}
    \int_{\mathcal{E}} \pi_{t}^{n} (\diff z) = \int_{\mathcal{E}} \pi_{n^{\alpha }t} (n^{\beta} \diff z) = \int_{n^{\beta}\mathcal{E}} \pi_{n^{\alpha }t } (\diff z)=\pi_{n^{\alpha }t} (n^{\beta}\mathcal {E})= 1, \; t\geq 0.
\end{equation*}

For each $n\in \mathbb{N}$, recalling the representation (\ref{rep_lambda_w.r.t_Poisson}) the formulation of the scaled intensity process $\{\lambda^{mH}_{n}(t), t\geq 0\}$ associated with marked Hawkes process $\{N^{mH}_{n}(t), t\geq 0 \}$ and the base intensity $\{\mu_n^{mH}(t), \geq 0\}$ is defined by 
\begin{equation}\label{def_intensity_scaled}
\lambda^{mH}_{n}(t)= \mu^{mH}_{n}(t)+ \int_{0}^t \int_{\mathcal{E}} \int_{0}^{\lambda^{mH}_{n}(s-)} \phi_{n} (t-s, z) P_{n}(\diff s, \diff z, \diff u),
\end{equation}
where $\phi_n(t-s,z) := \phi (n^{\alpha}(t-s), n^{\beta}z)$, and $P_{n}(\diff s, \diff z, \diff u) $ is an independent Poisson measure with intensity measure $n^{\alpha} \diff s \pi_{n^{\alpha}s}(n^{\beta}\diff  z) \diff u$. 
As a consequence of definition (\ref{def_N^{mH}_measure scaled}), the Poisson random measure $P_{n}(\diff s, \diff z, \diff u) $ is scaled in the following manner: 
\begin{equation}\label{def_poisson_measure_scaled}
P_{n}(\diff s, \diff z, \diff u):= P (n ^{\alpha} \diff s, n^{\beta}\diff z, \diff u)
\end{equation}
such that the compensated Poisson measure $P_{n}(\diff s, \diff z, \diff u)-n^{\alpha} \diff s \pi^{n}_{s}(\diff z) \diff u$ is a martingale measure.

Based on the scaling (\ref{def_intensity_scaled}), it follows from (\ref{def_E(lambda^{mH}(t))}) that
the expected value $\E( \lambda_{n}^{mH}(t))$ has the following representation 
 \begin{align*}
     \mathbb{E} (\lambda^{mH}_{n}(t)) = &\mu_{n}^{mH}(t)+
     \int_0^t \int_{\mathcal{E}} n^{\alpha} \phi_{n} (t-s,z)\mathbb{E} (\lambda^{mH}_{n}(s)) \pi_{s}^{n}(\diff z) \diff s \\
     =&\mu_{n}^{mH}(t)+
     \int_0^t \int_{n^{\beta}\mathcal{E}} n^{\alpha} \phi (n^{\alpha}t-n^{\alpha}s,z)\E(\lambda^{mH}_{n}(s)) \pi_{n^{\alpha}s}(\diff z) \diff s \\
    =&  \mu_{n}^{mH}(t)+
     \int_0^t  n^{\alpha} \psi(n^{\alpha}t,n^{\alpha}s)\E(\lambda^{mH}_{n}(s)) \diff s .
 \end{align*}
 It leads to the following scaled kernel $\phi_{n}(t,s)$, weighted kernel $\psi_{n}(t,s)$ and resolvent kernel $\Gamma_{n}(t,s)$ defined as
 \begin{align} 
 &\ol{\phi}_{n}(t,z) := n^{\alpha} \phi_{n}(t,z)= n^{\alpha} \phi (n^{\alpha}t, n^{\beta} z), \quad t\geq 0, \quad z \in \mathcal{E}, \notag \\
 & \ol{\psi}_{n}(t,s) := n^{\alpha}\psi (n^{\alpha} t, n^{\alpha}s) , \quad \ol{\Gamma}_{n}(t,s):= n^{\alpha}\Gamma(n^{\alpha}t, n^{\alpha} s), \quad 0 \leq s \leq t. \label{def_scaled_phi_psi_gamma_n}
 \end{align}
 Therefore, in view of (\ref{def_expected_intensity_w.r.t gamma}) we obtain
 \begin{equation}\label{def_E(lambda_n(t))}
     \mathbb{E} \left( \lambda^{mH}_{n}(t) \right) =
     \mu^{mH}_{n}(t) + \int_0^t \ol{\Gamma}_{n}(t,s) \mu^{mH}_{n}(s) \diff s.
 \end{equation}
 \begin{remark}
    
     The scaling factors $\alpha$ and $\beta$ introduced in the present study demonstrate that at which magnitude the temporal dynamics (\ref{def_N^{mH}_scaled}) of the marked Hawkes process $\{N^{mH}_{n}(t), t\geq 0\}$ evolves simultaneously with respect to time and mark in the asymptotic regime, rather than at the conventional rate $n$. In particular, $\alpha$ controls how rapidly the original dynamics evolves as $n\ra \infty$ depending on $\alpha<1,=1,>1$, which means either system time horizon increases slowly compared to rate $n$ (\textit{sublinear}) or grows linearly at rate $n$ (\textit{linear}), or accelerates much faster than rate $n$ (\textit{superlinear}). As a result, incorporation of $lapha$ originates the multi-scaling high-intensity asymptotic regime, at which the base intensity $\mu^{mH}(.)$ and the intensity $\lambda^{mH}_{n}(.)$ increases at rate $n^{\alpha}$, and hence the cumulative intensity $\int_0^{n^{\alpha}t} \lambda^{mH}(s)\diff s = \mathcal{O}(n^{\alpha})$. Moreover, the excitation effect on the intensity (\ref{def_intensity_scaled}) determined by the kernel $\phi_n(t,z)$ asymptotically concentrates towards the origin at rate $n^{\alpha}$ either slowly $(\alpha<1)$, or linearly $(\alpha=1)$, or faster $(\alpha>1)$. The important feature of this analysis we have discussed in the following Section \ref{section_FLLN_theorems} that the cumulative scaled kernel converges, away from the origin, to its total mass in Lemma \ref{lemma_conv_phi_wrt_measure}, and the limiting kernel $\Phi(z)$ is therefore not an ordinary spread-out function; it exhibits a Dirac-type mass concentration at zero elapsed time, and so does the weighted kernel $\psi(s,v)$ and resolvent kernel $\Gamma_n(s,v)$.
     
 \end{remark}
 \begin{remark}
     The mark scaling parameter $\beta (\geq 0)$ plays a substantial role to control that excitation of $\{N^{mH}_{n}(t), t\geq 0\}$ and the marks distribution $\pi^{n}_{t}(\mathcal{A})$ and their influence on the intensity measure via the kernel $\phi_n(t,z)$. The invariance property (\ref{invariant_property_mark_space}) ensures that the mark transformation does not change the underlying mark space. The very common class satisfying (\ref{invariant_property_mark_space}) are the Euclidean space $\mathbb{R}^d$, or $\mathbb{R}^{d}_{+}$, $d\geq 1$. For any bounded domain, e.g., $[a,b]^{d}$, $0\leq a < b \leq 1$, it is preferable to consider $\beta=0$, otherwise \ref{invariant_property_mark_space}) may not hold. However, $\mathcal{E}$ is $\beta$-invariant, still it may strongly influence the 
     effective excitation created by an event, and therefore the stability and resolvent behavior of the Hawkes process.
     
 \end{remark}

\subsection{Shot noise process under multi-scaling regime}
Next, we consider the scaled shot noise process $\{Y^{mH}_{n}(t), t\geq 0\}$ given by 
\begin{equation*}
Y_{n}^{mH}(t):=\sum_{i=1}^{N^{mH}_{n}(t)} \theta (t-\tau_{i,n}, \eta_{i,n}), \q t\geq 0.
\end{equation*}
It follows from definitions (\ref{def_Y^{mH}_int}) and (\ref{def_N^{mH}_scaled}) that the stochastic integral representation of $Y^{mH}_{n}(t)$ is given by 
\begin{equation}\label{def_Y^{mH}_{n}(t)}
Y^{mH}_{n} (t):=\int_0^{t} \int_{\mathcal E} \theta_{n}(t-s,z) 
 N^{mH}(n^{\alpha} \diff s,  n^{\beta}\diff z) 
=\int_0^{t} \int_{\mathcal E} \theta_{n}(t-s,z) N^{mH}_{n}(\diff s, \diff z),
\end{equation}
where $\theta_{n}(t-s,z) := \theta(n^{\alpha}t-n^{\alpha} s, n^{\beta} z)$.

\subsection{Functional Law of Large Numbers}\label{section_FLLN_theorems}
For every $\mathscr{B}_{\mathcal E}$-measurable real-valued function $f: \mathcal E\ra \mathbb{R}$, we define
\begin{equation*}
N_{n,t}^{mH}(f)=\int_0^t \int_{\mathcal E} f(z) N^{mH}_{n}(\diff s, \diff z)=\int_0^t \int_{\mathcal E} f(z) N^{mH}(n^{\alpha} \diff s,  n^{\beta} \diff z).
\end{equation*}
In this section, we define the rescaled processes as follows:
\begin{equation}\label{def_fluid_scaled_N_{n}(t), Y_{n}(t)}
\ol{N}^{mH}_{n}(t):=\frac{1}{n^{\alpha}} N^{mH}_{n}(t), \q 
\ol{Y}^{mH}_{n}(t):=\frac{1}{n^{\alpha}} Y^{mH}_{n}(t);
\end{equation}
In order to prove the functional laws of large numbers, we assume the first order deterministic convergences on the given base intensity function and the kernel function and the mark distribution, as a consequences of  high intensity scaling regime.
\begin{assumption}\label{assm_base_intensity_FLLN}
	There exists a locally integrable function $\ol{\mu}^{mH}(.)$ such that 
	\begin{equation*}
	\sup_{0\leq t \leq T} \left | \int_0^t  \mu^{mH}_{n}(s) \diff s- \int_0^t \ol{\mu}^{mH}(s)\diff s\right| \ra 0, \q \xt{as} \q n\ra \infty,
	\end{equation*}
 which is equivalent to say that 
	\begin{equation*}
    \sup_{0\leq t \leq T} \left |\frac{1}{n^{\alpha}} \int_0^{n^{\alpha }t}   \mu^{mH}(s) \diff s- \int_0^t \ol{\mu}^{mH}(s)\diff s\right| \ra 0, \q \xt{as} \q n\ra \infty.
     \end{equation*}
Additionally, by Assumption \ref{assm_initial_intensity_bound},  
     \begin{equation*}
         \sup_{n\geq 1} \sup_{0\leq t \leq T} |\mu^{mH}_{n}(t)|<\infty.
     \end{equation*}
\end{assumption}

\begin{assumption}\label{assm_phi_conv_LLN}
    For every $\delta>0$, and $T>0$, under Assumption \ref{assm_kernel_locally_bounded}, there exists $\Phi(z)$ defined by 
    \begin{equation}\label{def_Phi_z_limit_phi_n_kernel}
        \Phi (z) :=\lim_{n\ra \infty} \int_0^{\infty} \phi(s, n^{\beta}z)  \diff s= \lim_{n\ra \infty} \bar {\Phi}(n^{\beta}z), \quad z\in \mathcal{E}
    \end{equation}
    such that 
    \begin{equation}\label{assm_phi_n_convergence_Phi}
        \sup_{\delta \leq t \leq T} \esp_{z\in \mathcal{E}} 
        \left| \int_0^t \ol{\phi}_{n}(s,z) \diff s - \Phi(z) \right| \ra 0 , \q \xt{as} \q n\ra \infty,
    \end{equation}
    provided the uniform tail condition (\ref{assm_phi_tail_cond}) holds.
In addition, we assume 
\begin{equation}\label{assm_phi_bound_LLN}
    \sup_{n\geq 1} \sup_{0\leq t\leq T}\esp_{z\in \mathcal{E}}
     \int_0^{t} \left|\ol{\phi}_{n}(s,z) \right|\diff s<\infty.
\end{equation}

\end{assumption}

\begin{note}
    Note that the convergence of the kernel $\ol{\phi}_{n}(.,.)$ in Assumption \ref{assm_phi_conv_LLN} holds in $[\delta, T]$ for any $0<\delta <T$, but not generally is uniform on $[0,T]$ as at $t=0$, $\int_0^t \ol{\phi}_{n}(s,z) \diff s =0$. Whereas the limiting value for $t>0$ is the total mass $\Phi (z)$. So there is a jump at time $t=0$ in the limiting cumulative kernel, which is a expected consequence of the kernel function $\phi$ for marked Hawkes process behaving as Dirac-type limit.
\end{note}
\begin{remark}
   The convergence in Assumption \ref{assm_phi_conv_LLN} is not a direct consequences of the scaling (\ref{def_scaled_phi_psi_gamma_n}) and Assumption \ref{assm_kernel_locally_bounded}. The decomposition of (\ref{assm_phi_n_convergence_Phi}) yields 
   \begin{align*}
      \sup_{\delta \leq t \leq T} \esp_{z\in \mathcal{E}}  \left|\int_0^t \ol{\phi}_{n}(s,z) \diff s - \Phi(z) \right|
      & \leq  \sup_{\delta \leq t \leq T} \esp_{z\in \mathcal{E}} 
       \left| \int_0^t \ol{\phi}_{n}(s,z) \diff s - \int_0^{\infty} \phi(s, n^{\beta}z) \diff s\right|\\
       & ~~~~~~~~~~~~~~~~~+ \esp_{z\in \mathcal{E}} 
       \left| \int_0^{\infty} \phi (s,n^{\beta}z) \diff s - \Phi(z)\right|.
   \end{align*}
   Here, the first convergence directly follows from (\ref{def_Phi_z_limit_phi}) and (\ref{assm_phi_tail_cond}). However, in order to get the convergence (\ref{assm_phi_n_convergence_Phi}), we additionally require the convergence in the mark scaling term $\phi(., n^{\beta}z)$, i.e.,
   \begin{equation}\label{assm_phi_convergence_z_scaling}
     \esp_{z\in \mathcal{E}} 
       \left| \bar{\Phi}(n^{\beta}z)- \Phi(z)\right| \ra 0,  \q \xt{as} \q n\ra \infty,
   \end{equation}
   under the invariance property of the mark space $\mathcal{E}$. Note that, it is not mathematically justified to say that $\alpha$ dominates $\beta$ or $\beta$ dominates $\alpha$, they only determine the magnitudes of temporal decay and the mark dependence on the limiting kernel.
\end{remark}

\begin{example}
Due to necessity of the convergence (\ref{assm_phi_convergence_z_scaling}), among the class of functions as discussed in Example \ref{example_phi_nonscaled}, we can particularly choose $h\in \mathbb{L}^{\infty}(\mathcal{E})$ such that $h(n^{\beta}z)$ is bounded and has a limit, say $h_{\infty}(z)$, at infinity uniformly. For example, 
$h(z)= 1+ \frac{1}{1+z}$, or $h(z) = \frac{z}{1+z}$, or $h(z)=\frac{1}{1+z^2}$, $z\in \mathcal{E}=(0, \infty)$. But the convergence may fail if in case of $h$ oscillates $(h(z)=sin(z))$, or explodes at infinity.
\end{example}

\begin{assumption}\label{assm_pi_FLLN}
There exists a probability measure $\ol{\pi}_{t}(\mathcal A)$ on $\mathcal{E}$ i.e. $(\ol{\pi}_{t}(\mathcal E)=1)$ such that for every bounded $\mathscr{B}_{\mathbb{R}_{+}} \otimes \mathscr{B}_{\mathcal{E}}$-measurable function $f$,
\begin{equation}
    \sup_{0\leq t \leq T} \left| \int_{\mathcal{E}} f(t, z) \pi^{n}_{t}(\diff z) - \int_{\mathcal{E}}f(t,z) \ol{\pi}_{t}(\diff z) \right| \ra 0, \quad \xt{as} \quad n \ra \infty.
\end{equation}
\end{assumption}
Instead of assuming the above weak convergence (\ref{assm_pi_FLLN}) of the mark distribution, for instance one can consider the stronger \textit{total variation convergence uniformly on $t$}, i.e.
\begin{equation}\label{assm_pi_FLLN_TV}
  \sup_{0\leq t \leq T}  \sup_{\mathcal A\in \mathscr{B}_{\mathcal E}} | \pi^{n}_{t} (\mathcal A) - \ol{\pi}_{t}(\mathcal A)| \ra 0 , \quad \xt{as} \quad n \ra \infty.
\end{equation}
It is clear from above definitions that the convergence (\ref{assm_pi_FLLN_TV}) implies Assumption \ref{assm_pi_FLLN}.

As a consequence of the scaling (\ref{def_scaled_phi_psi_gamma_n}), for every $\delta >0 $ and $T>0$, Assumptions \ref{assm_phi_conv_LLN} and \ref{assm_pi_FLLN} together yield the following lemma. 
\begin{lemma}\label{lemma_conv_phi_wrt_measure}
Under Assumptions  \ref{assm_phi_conv_LLN} and \ref{assm_pi_FLLN}, for every $\delta>0$ and $T>0$,
\begin{equation}\label{assm_phi_wrt_measure_FLLN}
     \sup_{\substack{0\leq v \leq t \leq T \\t-v \geq \delta}} \left|
     \int_0^t \int_{\mathcal{E}} \ol{\phi}_{n}(s-v,z) \pi_{v}^{n}(\diff z) \diff s - \int_{\mathcal{E}}\Phi(z) \ol{\pi}_{v}(\diff z) \right| \ra 0,  \quad \xt{as} \quad n \ra \infty.
\end{equation}
\end{lemma}

\noindent
\textbf{Convergences of $\ol{\psi}_{n}(t,s)$ and $\ol{\Gamma}_n(t,s)$:}
For the scaled weighted kernel $\ol{\psi}_{n}(t,s)$, the definition (\ref{def_scaled_phi_psi_gamma_n}) yields 
\begin{equation*}
    \ol{\psi}_{n}(t,s)= \int_{\mathcal{E}} \ol{\phi}_{n}(t-s,z) \pi_{s}^{n}(\diff z),\q  t\geq s \geq 0,
\end{equation*}
where the scaled kernel $\ol{\phi}_{n}(t,z)$ is supported on $\mathbb{R}_{+}$, i.e., $\ol{\phi}_n(t,z)=0$ for $t<0$. Consequently $\ol{\psi}_n(t,s)=0$ for $t<s$. It follows from Lemma \ref{lemma_conv_phi_wrt_measure} that for every $\delta>0$ and $T>0$,
\begin{equation}\label{assm_psi_LLN}
   \sup_{\substack{0\leq v \leq t \leq T \\t-v \geq \delta}} \left|\int_{0}^{t} \ol{\psi}_{n}(s,v) \diff s - \Psi(v) \right|,\ra 0,  \quad \xt{as} \quad n \ra \infty.
\end{equation}
where $\Psi(v):= \int_{\mathcal{E}} \Phi(z) \ol{\pi}_{v}(\diff z)$.

Similarly, for scaled resolvent kernel $\ol{\Gamma}_{n}(t,s)$, we get the convergence 
\begin{equation}\label{assm_gamma_LLN}
   \sup_{\substack{0\leq v \leq t \leq T \\t-v \geq \delta}} \left|\int_{0}^{t} \ol{\Gamma}_{n}(s,v) \diff s - \frac{\Psi(v)}{1-\Psi(v)} \right| \ra 0,  \quad \xt{as} \quad n \ra \infty,
\end{equation}
for every $\delta>0$ and $T>0$, provided 
\begin{equation}\label{psi_stability_LLN}
    \sup_{0\leq v \leq T}\Psi(v)<1,
\end{equation}
for every $T>0$. Here, the stability condition (\ref{psi_stability_LLN}) is direct consequence of (\ref{psi_subcritical_condition}).
It is evident from the above convergence relation (\ref{assm_gamma_LLN}) along with (\ref{assm_phi_bound_LLN}) that 
\begin{equation}\label{assm_gamma_n_bound}
    \sup_{n\geq 1} \sup_{0\leq v \leq t \leq T } \int_0^t \left|\ol{\Gamma}_{n}(s,v) \right| \diff s  < \infty.
\end{equation}
\begin{remark}
    From \ref{assm_psi_LLN}, it is straightforward that the limiting resolvent mass is the summable series $\sum_{k=1}^{\infty} \Psi(v)^{k}$, and the stability assumption $\sup_{0\leq v \leq T}\Psi(v)<1$ yields that this limit $\frac{\Psi(v)}{1-\Psi(v)}$ is a singular resolvent measure, a Dirac-type kernel concentrated on the diagonal at $s=v$. 
\end{remark}

Before stating the limiting theorems for marked Hawkes processes under the multi-scaling regime, we need to identify the limiting behaviour of the intensity process and analyze it. We introduce the integral process $\{\ol{N}^{mH}(t), t\geq 0\}$, where
\begin{equation}\label{def_LLN_N^{mH}_limit}
\overline{N}^{mH}(t):= \int_0^t \frac{\ol{\mu}^{mH}(s) }{1-\Psi(s)} \diff s, \q t \geq 0.
\end{equation}
For our convenience, we additionally define 
\begin{equation}
    C^{mH}(t):= \frac{\ol{\mu}^{mH}(t) }{1-\Psi(t)}, \q t\geq 0,
\end{equation}
so that $\overline{N}^{mH}(t)= \int_0^t C^{mH}(s) \diff s$.

Our aim is to establish that the deterministic integral process $\{\ol{N}^{mH}(t), t\geq 0\}$ defined above is the limiting process of $\{\ol N^{mH}_{n}(t), t\geq 0\}$ as $n\ra \infty$ with trajectories in $\md ([0,\infty), \mathbb{R})$. Before that, we need to show that the cumulative expected intensity process given by $\{\int_0^t \E (\lambda_{n}^{mH}(t)), t\geq 0\}$ converges to $\{\ol N^{mH}_{n}(t), t\geq 0\}$ deterministically, which is presented in next Lemma \ref{lemma_intensity_LLN_expected}  with proof in Section \ref{section_proofs of auxiliary results}.

\begin{lemma}\label{lemma_intensity_LLN_expected}
	Suppose Assumptions \ref{assm_initial_intensity_bound}, \ref{assm_kernel_locally_bounded} and \ref{assm_base_intensity_FLLN}, \ref{assm_phi_conv_LLN}, \ref{assm_pi_FLLN} are satisfied. Then
	\begin{equation*}
	\lim_{n\ra \infty}   \sup_{0\leq t \leq T}\left | \int_0^t \E \left( \lambda_{n}^{mH}(s) \right)\diff s-    \int_0^t C^{mH}(s) \diff s \right| =0. 
	\end{equation*}
\end{lemma}

    Next thing is to show that the rescaled cumulative intensity process defined by $\{\int_0^t \lambda ^{mH}_{n}(s)\diff s, t\geq 0\}$ converges to the deterministic entity $\{\ol N^{mH}_{n}(t), t\geq 0\}$ in probability uniformly on compact sets as $n\ra  \infty $ in $\mathbb{C} ([0,\infty), \mathbb{R})$. For this purpose, by Proposition \ref{prop_intensity_martingale}, and definitions (\ref{def_intensity_scaled}) (\ref{def_E(lambda_n(t))}), we consider the difference
    \begin{align*}
    &\int_0^t \lambda^{mH}_{n} (s) \diff s- \int_0^t \E \left( \lambda^{mH}_{n}(s) \right) \diff s\\
    =
    & \int_0^{t} \int_0^{s} \int_{\mathcal{E}} \int_{0}^{\lambda_{n} ^{mH}(v-)}
	\phi_{n}(s-v, z) M^{P}_{n}(\diff v, \diff z, \diff u) \diff s\\
	&+  \int_0^{t} \int_0^s \Gamma_{n} (s,v) \int_0^v \int_{\mathcal{E}} \int_{0}^{\lambda_{n}^{mH}(w-)} \phi_{n} (v-w, z) M^{P}_{n}(\diff w, \diff z, \diff u) \diff v \diff s,
    \end{align*}
    where 
    \begin{equation*}
        M^{P}_{n} ( \diff s \diff z , \diff u)= M^{P} ( n^{\alpha}\diff s, n^{\beta} \diff z, \diff u).
    \end{equation*}
    After the change of variables, we obtain
    \begin{align*}
    &\int_0^t \lambda^{mH}_{n} (s) \diff s- \int_0^t \E \left( \lambda^{mH}_{n}(s) \right) \diff s \notag \\ 
    &=    \frac{1}{n^{\alpha}}  \int_0^{n^{\alpha}t} \int_0^{s} \int_{n^{\beta} \mathcal{E}} \int_{0}^{\lambda^{mH}(v-)}
	\phi(s-v, z) M^{P}(\diff v, \diff z, \diff u) \diff s
    \notag \\
	&+   \frac{1}{n^{\alpha}} \int_0^{n^{\alpha} t} \int_0^s \Gamma (s,v) \int_0^v \int_{n^{\beta}\mathcal{E}} \int_{0}^{\lambda^{mH}(w-)} \phi (v-w, z) M^{P}(\diff w, \diff z, \diff u) \diff v \diff s.
    \end{align*}

	Define 
	\begin{align}
	X_{n,1}(t)=&\int_0^{n^\alpha t} \int_0^{s} \int_{n^{\beta}\mathcal{E}} \int_{0}^{\lambda ^{mH}(v-)}
	\phi(s-v, z) M^{P}(\diff v, \diff z, \diff u) \diff s,
	\label{def_X_{n,1}}\\
	X_{n,2}(t)=&
	\int_0^{n^\alpha t} \int_0^s \Gamma (s,v) \int_0^v \int_{n^{\beta}\mathcal{E}} \int_{0}^{\lambda^{mH}(w-)}
	\phi (v-w, z)
	M^{P}(\diff w, \diff z, \diff u) \diff v \diff s\label{def_X_{n,2}}
	\end{align}
    such that 
    \begin{equation}\label{sum_of_X_{n,1}_X_{n,2}}
        \int_0^t \lambda^{mH}_{n} (s) \diff s- \int_0^t \E \left( \lambda^{mH}_{n}(s) \right) \diff s 
        = \frac{1}{n^{\alpha}} \left( X_{n,1}(t) + X_{n,2} (t)\right).
    \end{equation}
	Applying Fubini's theorem to change the order of integration (\ref{def_X_{n,1}}), (\ref{def_X_{n,2}}) and $n^{\beta} \mathcal{E}= \mathcal{E}$, we obtain 
	\begin{align}
	X_{n,1}(t)=&\int_0^{n^{\alpha}t}\int_{\mathcal E}\int_{0}^{\lambda^{mH}(v-)}  \int_{v}^{n^{\alpha}t} 
	\phi(s-v,z) M^{P}(\diff v, \diff z, \diff u) \diff s, \label{def_new_X_{n,1}}
	\\
	X_{n,2}(t)=& \int_0^{n^{\alpha}t} \int_{\mathcal E}\int_0^{\lambda^{mH}(v-)} \int_{v}^{n^{\alpha}t }
	\int_{w}^{n^{\alpha }t}
	\Gamma(s,w) \phi(w-v,z) M^{P}(\diff v, \diff z, \diff u) \diff w \diff s. \label {def_new_X_{n,2}}
	\end{align}
	By (\ref{def_new_X_{n,1}}) and (\ref{def_new_X_{n,2}}), 
	\begin{align}\label{diff_intensity_FLLN}
	&X_{n,1}(t)+X_{n,2}(t)\notag \\
	&=\int_0^{n^{\alpha}t} \int_{\mathcal E}\int_{0}^{\lambda^{mH}(v-)}
	M^{P}(\diff v, \diff z, \diff u) 
	\int_v^{n^{\alpha}t} \bigg( \phi(w-v,z)+\int_{w}^{n^{\alpha}t}\Gamma (s,w) \phi(w-v,z) \diff s
	\bigg) \diff w.
	\end{align}	
	Following the definition (\ref{def_tau_t(v,z)_function}), we consider the scaled entity $\mathcal T_{n^\alpha t}(.,.)$ on $\mathbb{R}_{+}\times \mathcal E$, where  
	\begin{equation}\label{def_tau_t(v,z)_scaled_function}
	\mathcal T_{n^{\alpha}t}(v,z)= \int_v^{n^{\alpha}t} \phi(w-v,z) \bigg(1	+\int_{w}^{n^{\alpha}t}\Gamma (s,w)  \diff s	\bigg) \diff w  .
	\end{equation}
	As a result, the representation (\ref{sum_of_X_{n,1}_X_{n,2}}) and (\ref{diff_intensity_FLLN}) yields
	\begin{equation*}
	\int_0^{t} \lambda_{n}^{mH}(s)\diff s-  \int_0^{t} \E \left( \lambda_{n}^{mH} (s) \right) \diff s= \frac{1}{n^{\alpha}} \int_0^{n^{\alpha}t} \int_{\mathcal E}\int_{0}^{\lambda^{mH}(v-)}
	\mathcal T_{n^{\alpha}t}(v,z)	M^{P}(\diff v, \diff z, \diff u). 
	\end{equation*}
	Taking $n\ra \infty$, by definition (\ref{def_tau_t(v,z)_scaled_function}) we obtain 
	\begin{align*}
	\lim_{n\ra \infty} \mathcal{T}_{n^{\alpha}t}(v,z)
    = & \int_v^{\infty}  \phi(w-v,z) \bigg(1 +\int_w^{\infty} \Gamma (s,w) \diff s\bigg) \diff w \notag \\
	&\leq  \esp_{z\in \mathcal E} \int_0^{\infty}
    |\phi(w,z)| \bigg(1+\int_0^{\infty}  \Gamma(s,w) \diff s \bigg) \diff w \notag \\
	&<\infty, 
	\end{align*}
	where the last inequality follows from Assumption \ref{assm_kernel_locally_bounded}.

     Using change of variables, it can be shown from definitions (\ref{def_tau_t(v,z)_function}) and (\ref{def_tau_t(v,z)_scaled_function}) via (\ref{def_scaled_phi_psi_gamma_n}) that 
    \begin{equation}\label{def_tau_t(v,z)_scaled_new}
      \mathcal T_{n^{\alpha}t}(v,z) = \int_v^t \ol{\phi}_{n}
      (w-v,z)\left(1 +\int_w^t \ol{\Gamma}_{n}(s,w)  \diff s \right) \diff w.
    \end{equation}
    Define 
        \begin{equation}\label{def_tau_infty}
         \mathcal{T}_{\infty}(v,z)=  \frac{\Phi(z)}{1-\Psi(v)}, \q v\geq 0, \q z\in \mathcal{E}.
         \end{equation} 
         One can show that $\mathcal T_{n^{\alpha}t}(v,z)$ converges to $\mathcal{T}_{\infty}(v,z)$ for every $t>0$ as $n \ra \infty$. This convergence is not only limited to point-wise. One can develop the stronger uniform convergence on $t$ and $z$.   
    The next lemma establishes that $\mathcal T_{n^{\alpha}t}(v,z)$ converges to $\mathcal{T}_{\infty}(v,z)$ uniformly on every compact sets provided $t-v\geq \delta$ for any $T>\delta >0$ under suitable conditions, which essentially leads us to define certain martingale process in order to prove the main results of the article.

     \begin{lemma}\label{lemma_tau_convergence_tau_infty}
         Under Assumptions \ref{assm_kernel_locally_bounded}, \ref{assm_phi_conv_LLN}, for every $\delta>0$ and $T>0$, the following assertions hold.
         \begin{enumerate}[(i)]
             \item 
             \begin{equation} \label{lemma_tau_limit_bound_1}
                 \sup_{n\geq 1}\sup_{ 0 \leq v \leq  t \leq T} \esp_{z\in \mathcal{E}} 
                 \left|\mathcal{T}_{n^{\alpha}t}(v,z) \right|< \infty.
             \end{equation}
             \item 
             \begin{equation}\label{lemma_tau_conv_2}
             \sup_{\substack{0 \leq v \leq  t \leq T\\ t-v \geq \delta}} \esp_{z\in \mathcal{E}} \left|\mathcal{T}_{n^{\alpha}t}(v,z) - \mathcal{T}_{\infty}(v,z) \right| \ra 0 , \q \text{as} \q n \ra \infty.
             \end{equation}         
         \end{enumerate}
     \end{lemma}

Knowing the characterization of  $ \mathcal{T}_{n^{\alpha}t}(.,.)$ as above, for  a  given $\mathscr{B}_{\mathcal E}$-measurable,  real valued function $f:\mathcal E\ra \mathbb{R}$,  we define the integral process associated with the martingale measure $ M^{P}(\diff v, \diff z, \diff u)$ as follows:
	\begin{equation}\label{def_Upsilon_{n}(t)}
	\Upsilon_{n,t}(f)= \int_0^{n^{\alpha}t} \int_{\mathcal E} \int_0^{\lambda^{mH}(v-)} f(z) \mathcal{T}_{n^{\alpha}t}(v,z) M^{P}(\diff v, \diff z, \diff u) , \q t\geq 0.
	\end{equation}

    However, the process $\{\Upsilon_{n,t}(f), t\geq 0\}$ is not an $\mathscr{F}_{n}$-martingale due to the presence of $\tau_{n^{\alpha}t}(v,z)$ in the integrand, as it changes with respect to time $t$. Therefore, in order to preserve the martingale property, we define 
	\begin{equation}\label{def_Upsilon_tilda_n}
	\widetilde{\Upsilon}_{n,t}(f)= \int_0^{n^{\alpha}t} \int_{\mathcal E} \int_0^{\lambda^{mH}(v-)} f(z) \mathcal{T}_{\infty}(v,z) M^{P}(\diff v, \diff z, \diff u) , \q t\geq 0,
	\end{equation}	
    for given bounded $\mathscr{B}_{\mathcal{E}}$-measurable function $f$ and $\mathcal{T}_{\infty}(v,z)$, independent of the terminal variable $t$. The following proposition shows that the latter process $\frac{1}{n^{\alpha/2}}\widetilde{\Upsilon}_{n,t}(f)$ is distributionally equivalent to $\frac{1}{n^{\alpha/2}} \Upsilon_{n,t}(f)$ uniformly on every compact intervals at diffusion scaling factor $\frac{1}{n^{\alpha/2}}$, which immediately results 
            \begin{equation*}
            \lim_{n\ra \infty} \mathbb{P} \left( \sup_{0 \leq t \leq T} \left| \frac{1}{n^{\alpha/2+\eta}}\Upsilon_{n,t}(f)- \frac{1}{n^{\alpha/2+\eta}}\widetilde{\Upsilon}_{n,t}(f) \right|>\epsilon \right) =0,
        \end{equation*}   
        for every $\eta> 0$.
    \begin{proposition}\label{prop_upsilon_tilde_upsilon_L_2}
        Suppose the assumptions of Lemma \ref{lemma_tau_convergence_tau_infty} are satisfied. Then, for every $T>0$, and $\epsilon>0$, 
        \begin{equation*}
            \lim_{n\ra \infty} \mathbb{P} \left( \sup_{0 \leq t \leq T} \left| \frac{1}{n^{\alpha/2}}\Upsilon_{n,t}(f)- \frac{1}{n^{\alpha/2}}\widetilde{\Upsilon}_{n,t}(f) \right| > \epsilon \right) =0 .
        \end{equation*}
    \end{proposition}
    As a consequence of Proposition \ref{prop_upsilon_tilde_upsilon_L_2}, the rest of the following convergences in this section will deal with the latter process $\{\widetilde{\Upsilon}_{n,t}(f), t\geq 0\}$. The next two results demonstrate the martingale property of $\{\widetilde{\Upsilon}_{n,t}(f), t\geq 0\}$ and its convergence, which play the key role in order to conclude the main convergence results in the later sections. The proofs of Lemma \ref{lemma_martingale_upsilon_{n}(t)} and Propositions \ref{prop_upsilon_tilde_upsilon_L_2},  \ref{prop_FLLN_Upsilon_{n}(t)} are given in Section \ref{section_proofs of auxiliary results} in detail.

    \begin{lemma}\label{lemma_martingale_upsilon_{n}(t)}
	Suppose Assumptions \ref{assm_kernel_locally_bounded}, \ref{assm_phi_conv_LLN} are satisfied. Then for every $\mathscr{B}_{\mathcal E}$-measurable function $f,g \in \mathbb{C}_{b}(\mathcal E, \mathbb{R})$, the integral process $\{\widetilde{\Upsilon}_{n,t}(f), t\geq 0\}$ is an $\mathscr{F}_{n}$-local square-integrable martingale with quadratic covariation process 
	$\{  \langle \widetilde{\Upsilon}_{n,.}(f), \widetilde{\Upsilon}_{n,.}(g) \rangle (t), t\geq 0\}$, where 
	\begin{equation*}
	   \langle \widetilde{\Upsilon}_{n,.}(f), \widetilde{\Upsilon}_{n,.}(g) \rangle (t)=
	   \int_0^{n^{\alpha}t} \int_{\mathcal E}  f(z) g(z) \mathcal{T}_{\infty}^{2}(v,z) \lambda^{mH}(v)  \pi_v(\diff z) \diff v, \q t\geq 0.
	    \end{equation*} 
	\end{lemma}
    The martingale nature of Lemma \ref{lemma_martingale_upsilon_{n}(t)} leads the convergence of $\{\widetilde{\Upsilon}_{n}(t), t\geq 0\}$ process defined by 
    \begin{equation*}
     \widetilde{\Upsilon}_{n}(t):= \widetilde{\Upsilon}_{n,t}(\bm{1}) =\int_0^{n^{\alpha}t} \int_{\mathcal E} \int_0^{\lambda^{mH}(v-)}  \mathcal{T}_{\infty}(v,z) M^{P}(\diff v, \diff z, \diff u) , \q t\geq 0.
     \end{equation*}
     The next convergence in Proposition \ref{prop_FLLN_Upsilon_{n}(t)} proves the functional laws of large number result for $\{\widetilde{\Upsilon}_{n}(t), t\geq 0\}$.
    \begin{proposition}\label{prop_FLLN_Upsilon_{n}(t)}
	     Let the assumptions of Lemma \ref{lemma_martingale_upsilon_{n}(t)} hold. Then, for any $\eta>0$,  the process $\{ \frac{1}{n^{\alpha/2 +\eta}}  \widetilde{\Upsilon}_{n}(t), t\geq 0\}$ trivially converges to zero in probability uniformly on compact sets, i.e., for every $\epsilon>0$ and $T>0$, 
 	\begin{equation*}
	\lim_{n\ra \infty} \prob  \left( \sup_{0\leq t\leq T} \frac{1}{n^{\alpha/2 +\eta}} \left| \widetilde{\Upsilon}_{n}(t)\right|>\epsilon   \right)=0 .
	\end{equation*}
    \end{proposition}

In the end, knowing the functional convergence of $\{\widetilde{\Upsilon}_{n}(t), t\geq 0\}$ in Propositions \ref{prop_upsilon_tilde_upsilon_L_2} and \ref{prop_FLLN_Upsilon_{n}(t)}, together with Lemma \ref{lemma_intensity_LLN_expected} we get the desired convergence of $\{\int_0^{t} \lambda^{mH}_{n}(s)\diff s, t\geq 0 \}$ to $\{\ol N^{mH}_{n}(t), t\geq 0\}$ in probability uniformly on compact sets in $\md([0,\infty), \mathbb{R})$. This result is stated in the following theorem.
	
	\begin{theorem}\label{thm_intensity_LLN}
		If Assumptions \ref{assm_initial_intensity_bound}, \ref{assm_kernel_locally_bounded} and  \ref{assm_base_intensity_FLLN}, \ref{assm_phi_conv_LLN}, \ref{assm_pi_FLLN} hold, then $\{ \int_0^{t} \lambda_{n}^{mH}(s)\diff s, t\geq 0 \}$ converges to $\{\ol N^{mH}_{n}(t), t\geq 0\}$ in probability uniformly on compact sets in $\md([0,\infty), \mathbb{R})$, i.e., 	
		for any $T>0$ and $\epsilon>0$, 
		\begin{equation*}
		\lim_{n\ra \infty} \prob \left(
		\sup_{0\leq t\leq T} \left| \int_0^t \lambda^{mH}_{n}(s)\diff s - \int_0^t C^{mH}(s) \diff s 
        \right| >\epsilon \right) =0.
		\end{equation*}

	\end{theorem}


Now, we state the functional law of large numbers result for marked Hawkes process under multi- scaling regime based on the previous convergences of Lemmas \ref{lemma_intensity_LLN_expected}, Proposition \ref{prop_FLLN_Upsilon_{n}(t)}. We provide the proofs of the following Theorem \ref{thm_LLN_N^{mH}(t)}, \ref{thm_LLN_Y^{mH}(t)} in Section \ref{section_FLLN_proofs} in detail. 

\begin{theorem}\label{thm_LLN_N^{mH}(t)}
	Suppose the assumptions of Theorem \ref{thm_intensity_LLN} are satisfied. Then, the rescaled process $\{ \ol{N}^{mH}_{n}(t), t\geq 0\}$ converges in probability to the deterministic process $\{\ol{N}^{mH}(t), t\geq 0\}$ uniformly on compact sets as $n\ra \infty$, i.e. 
	for every $T>0$ and $\epsilon>0$,
    \begin{equation*}
	\lim_{n\ra \infty} \prob \left(\sup_{0\leq t \leq T}
	\left| \ol{ N} ^{mH}_{n}(t) - \ol{N}^{mH}(t) \right|>\epsilon  \right)=0.
	\end{equation*}

\end{theorem}

\begin{remark}\label{remark_LLN_i.i.d_mark}
	If the mark random sequence $\{\xi_{i}(\tau_{i}), i\in \mathbb{N}\}$ are i.i.d., then in view of Remark \ref{remark_iid of mark sequence}  $\psi(t,s)=\psi(t-s)$ and $\Gamma (t,s)=\Gamma(t-s)$, (see \cite{horst2021functional}. Therefore, the limiting process $\{\ol{N}^{mH}(t), t\geq 0\}$ turns out to be 
	\begin{equation*}
	\overline{N}^{mH}(t)=\int_0^t \mu^{mH}(s) \diff s+  \int_0^t \Gamma \star \mu^{mH} (s) \diff s.
	\end{equation*}
In addition, if $\mu^{mH}(t)=\mu^{mH}$, under i.i.d. framework, the resulting process of functional laws of large number yields 
\begin{equation*}
\overline{N}^{mH}(t)= \mu^{mH} t(1+||\Gamma||_{\mathbb{L}^{1}})= 	\frac{\mu^{mH} t}{1-||\psi||_{\mathbb{L}^{1}}}, 
\end{equation*} 
by Remark \ref{remark_iid of mark sequence}, knowing that $||\psi||_{\mathbb{L}^{1}}<1$.
\end{remark}

\begin{remark}\label{remark_LLN_simple Hawkes}
	For simple Hawkes process, the kernel function is of the form $\phi(t-s,z)=\phi(t-s)$, which implies $\psi(t-s)=\phi(t-s)$. If the base intensity function is assumed to be constant, i.e.,  $\mu^{mH}(t)=\mu^{mH}$, then
	\begin{equation*}
	\overline{N}^{mH}(t)=
	\frac{\mu^{mH} t}{1-||\phi||_{\mathbb{L}^{1}}}
	\end{equation*}
	where 
	\begin{equation*}
	\Gamma(t)=\phi \star \Gamma(t)=\Gamma \star \phi (t),
	\end{equation*}
	implying $||\Gamma||_{\mathbb{L}^{1}}=\frac{||\phi||_{\mathbb{L}^{1}}}{1-||\phi||_{\mathbb{L}^{1}}}$ by Remark \ref{remark_iid of mark sequence}. 
\end{remark}

Now, based on the deterministic approximations of the marked Hawkes process by Theorem \ref{thm_LLN_N^{mH}(t)}, we can obtain the functional laws of large numbers for the additive functional process $\{Y^{mH}_{n}(t), t\geq 0\}$ knowing the deterministic convergence 
of $\theta_{n}(t,z)=\theta(n^{\alpha}t, n^{\beta}z)$, with $\esp_{z\in \mathcal{E}} \theta_{n}(t,z)<\infty$, which is often referred to as the shot noise process, see \cite{horst2021functional}, \cite{pang2018functional}etc. To prove the following functional convergence, we require the following condition for the associated kernel function $\theta_{n}(t,z)$.
	\begin{assumption}\label{assm_theta_FLLN}
	There exists a locally bounded function $\ol{\theta}(.,.)$ on $\mathbb{R}_{+}\times \mathcal E$ such that 
    \begin{equation}
     	\sup_{0\leq t \leq T} \esp_{z\in \mathcal{E}}\left|  \theta_{n}(t, z) -  \ol{\theta}(t,z)\right| \ra  0, \q \xt{as} \q n\ra \infty.   
    \end{equation}
\end{assumption}
Here, Assumption \ref{assm_theta_FLLN} via Assumption \ref{assm_pi_FLLN} ensures that  
	\begin{equation}
	\sup_{0\leq t \leq T} \left| \int_0^t \int_{\mathcal{E}}  \theta_{n}(t-s, z) \pi_{s}^{n}(\diff z)  \diff s - \int_0^t \int_{\mathcal{E}} \ol{\theta}(t-s,z) \ol{\pi}_{s}(\diff z) \diff s\right| \ra  0, \q \xt{as} \q n\ra \infty.
	\end{equation}
Also
\begin{equation}
    \psi_{\theta} (n^{\alpha} t,n^{\alpha} s)= \int_{\mathcal{E}}  \theta_{n}(t-s, z) \pi_{s}^{n}(\diff z)  < \infty,  \quad 0 \leq s \leq t ,
\end{equation}
where $\psi_{\theta}(t,s) = \int_{\mathcal{E}}  \theta (t-s, z) \pi_{s}(\diff z)$ is the weighted kernel associated with $\theta(t,z)$.
\begin{theorem}\label{thm_LLN_Y^{mH}(t)}
	Suppose assumptions of Theorem \ref{thm_intensity_LLN} are satisfied. In addition, the convergence of Assumption \ref{assm_theta_FLLN} holds. Then the rescaled shot noise process $\{\ol{Y}^{mH}_{n}(t), t\geq 0\}$ converges to $\{\ol{Y}^{mH}(t), t\geq 0\}$ in probability uniformly on compact sets as $n\ra \infty$,	i.e., for every $T>0$ and $\epsilon>0$,
	\begin{equation}
	\lim_{n\ra \infty}\prob \left( \sup_{0\leq t \leq T} 
	\left|\ol{Y}^{mH}_{n}(t)- \overline {Y}^{mH}(t)\right| >\epsilon \right ) =0,
	\end{equation}
	where $\overline{Y}^{mH}(t)=
	\int_0^t \int_{\mathcal{E}} \ol{\theta}(t-s, z) C^{mH}(s)
    \ol{\pi}_{s}(\diff z) \diff s$.
\end{theorem}

%


\subsection{Functional Central Limit Theorems}\label{section_FCLT_theorems}
In this section, we focus on the functional central limit theorems for marked Hawkes process, additive functional process driven by the marked Hawkes measure under the asymptotic regime, which are built on the diffusion approximations of the intensity process. In particular, we demonstrate the Gaussian approximations of the cumulative intensity process in Theorem \ref{thm_FCLT_intensity_process} after centering and normalized the scaled process $\{\int_0^t \lambda ^{mH}_{n}\diff s, t\geq0\}$ under appropriate assumptions. The next two Theorems \ref{thm_FCLT_N^{mH}(t)} and \ref{thm_FCLT_Y_n(t)_process} shows the weak convergence of $\{N^{mH}_{n}(t), t\geq 0\}$ and $\{Y^{mH}_{n}(t), t\geq 0\}$ with trajectories in $\md([0, \infty), \mathbb{R})$. The complete proofs of this section is provided in Section \ref{section_FCLT_proofs}.


Define the diffusion scaled processes 
\begin{align}
\nonumber \widehat{N}^{mH}_{n}(t)& =N^{1-\delta}\left(\ol{N}^{mH}_{n}(t)-\overline{N}^{mH}(t)\right) , \\
\widehat{Y}_{n}^{mH}(t) & =N^{1-\delta}\left(\ol{Y}^{mH}_{n}(t)-\overline{Y}^{mH}(t)\right) ,\q 1/2 \leq \delta <1,
 \end{align}
and the diffusion scaled cumulative intensity process by 
\begin{equation}\label{def_Z_mH_diffusion}
\widehat{Z}^{mH}_{n}(t)= n^{\alpha/2} \left( \int_0^t
\lambda_{n}^{mH}(s)\diff s-\int_{0}^{t}C^{mH}(s) \diff s\right) , \q t\geq 0.
\end{equation}

The functional central limit theorems for these diffusion scaled processes are stated below. They are established using second-order deterministic convergences that build on Assumption \ref{assm_base_intensity_FLLN} and convergence relation (\ref{assm_gamma_LLN}), resulting from functional laws of large number scaling.

\begin{assumption}\label{assm_base_intensity_FCLT}
	Given $\ol{\mu}^{mH}(.)$ in Assumption \ref{assm_base_intensity_FLLN}, there exists a locally integrable function $\widehat{\mu}^{mH}(.)$ such that 
	\begin{equation}
	\sup_{0\leq t \leq T} \left| n^{1-\delta} \left(   \int_0^t  \mu^{mH}_{n}(s) \diff s- \int_0^t \ol{\mu}^{mH}(s)\diff s \right) - \int_0^t \widehat{\mu}^{mH}(s)\diff s\right|\ra 0, \q \xt{as} \q n\ra \infty.
	\end{equation}
	\end{assumption}

\begin{assumption}\label{assm_pi_FCLT}
    Given the measure $\ol{\pi}_{t}(\mathcal{A})$ on $\mathcal{E}$ in Assumption \ref{assm_pi_FLLN}, there exists a measure $\widehat{\pi}_{t}(\mathcal{A})$ on $\mathcal{E}$ such that for every bounded $\mathscr{B}_{\mathbb{R}_{+}} \otimes \mathscr{B}_{\mathcal{E}}$-measurable function $f(.,.)$,
    \begin{equation}
        \sup_{0\leq  t\leq T} \left|
        n^{1-\delta} \left(\int_{\mathcal E} f(t,z) \pi^{n}_{t}(\diff z) - \int_{\mathcal E} f(t,z) \ol{\pi}_{t}(\diff z)  \right)   
        - \int_{\mathcal{E}} f(t,z) \widehat{\pi}_{t}(\diff z) 
        \right| \ra 0, \q \xt{as} \q n\ra \infty. 
    \end{equation}
\end{assumption}

   \begin{assumption}\label{assm_Gamma_FCLT}
   Given $\ol{\Gamma}(.,.)$ in (\ref{assm_gamma_LLN}), there exists a locally bounded function $\widehat{\Gamma}(.,.)$ on 
    $\mathbb{R}_{+} \times \mathbb{R}_{+}$ such that for every $T>0$ and $\eta>0$, 
    \begin{equation*}
        \sup_{\substack{	0\leq v \leq t \leq T\\t-v \geq \eta }} \left|n^{1-\delta}\left(   \int_{v}^{t} \ol{\Gamma}_{n}(s,v) \diff s - \frac{\Psi(v}{1-\Psi(v)} \right)- \widehat{\Gamma}(t,v)  \right| \ra 0, \q \xt{as} \q n\ra \infty.
    \end{equation*}   
   \end{assumption}
    \begin{note}
        By virtue of (\ref{assm_psi_LLN}), if the weighted kernel $\psi(t,s)$ has second order deterministic approximation at order $n^{-(1-\delta)}$ given by 
        \begin{equation*}
            \int_{v}^{t} \ol{\psi}_{n}(s,v) \diff s = \Psi(v) + n^{-(1-\delta)} \widehat{H}_{\Psi}(t,v) + o (n^{-(1-\delta)}),
        \end{equation*}
        for some locally bounded function $\widehat{H}_{\Psi}(.,.)\in \mathbb{R}_{+} \times \mathbb{R}_{+}$, provided $t-v\geq \eta>0$, for any $\eta>0$, it results
            \begin{equation*}
            \int_{v}^{t} \ol{\Gamma}_{n}(s,v) \diff s = \frac{\Psi(v)}{1-\Psi(v)} + n^{-(1-\delta)} \frac{\widehat{H}_{\Psi}(t,v)}{\left(1-\Psi(v) \right)^{2}} + o (n^{-(1-\delta)}),
            \end{equation*}
        which essentially leads to Assumption \ref{assm_Gamma_FCLT} with $\widehat{\Gamma}(t,v)=\frac{\widehat{H}_{\Psi}(t,v)}{\left(1-\Psi(v) \right)^{2}}$, provided the stability condition (\ref{psi_stability_LLN}) holds.
    \end{note}

Before stating the diffusion approximations for $\{\widehat{Z}^{mH}_{n}(t), t\geq 0\}$, we analyze the diffusion scaled process $\{\widehat{\Upsilon}_{n,t}(f), t\geq 0\}$ for every $\mathscr{B}_{\mathcal E}$-measurable function $f$, which is defined by 
\begin{equation}\label{def_Upsilon_{n}_(t)_diffusion}
\widehat{\Upsilon}_{n,t}(f) =\frac{1}{n^{\alpha/2}}
\int_0^{n^{\alpha}t} \int_{\mathcal E} \int_0^{\lambda^{mH}(v-)} f(z) \mathcal T_{n^{\alpha}t}(v,z) M^{P}(\diff v, \diff z, \diff u)
\end{equation}
Note that, for every $\mathscr{B}_{\mathcal E}$-measurable function $f$,  $\{\widehat{\Upsilon}_{n,t}(f), t\geq 0\}$ is a normalized version of the process $\{\Upsilon_{n,t}(f), t\geq 0\}$ by normalizing the difference $\int_0^{t} \lambda_{n} ^{mH}(s) \diff s- \int_0^{t} \E \lambda_{n} ^{mH}(s) \diff s $ by $n^{\alpha/2}$.

\begin{proposition}\label{prop_fclt_upsilon_{n}(t)}
	Suppose the assumptions of Lemma \ref{lemma_martingale_upsilon_{n}(t)} are satisfied. Then the diffusion scaled process $\{\widehat{\Upsilon}_{n}(t), t\geq 0\}$ converges in distribution to the Gaussian process $\{\mathcal B (t), t\geq 0\}$ with trajectories in $\md([0, \infty), \mathbb{R})$ as $n\ra \infty$, where the limiting process $\{\mathcal B(t), t\geq 0\}$ is a mean zero Gaussian process with covariance 
	\begin{equation}\label{limit_B(t)_covariance}
    \E \mathcal {B}(t)\mathcal B(s)=
	\int_0^{t\wedge s} \int_{\mathcal E}  \mathcal T_{\infty}^{2}(v,z) C^{mH}(v)  \ol{\pi}_v(\diff z)\diff v .
	\end{equation}
	\end{proposition}

Now, we proceed to state that the weak convergence of the cumulative intensity process  
$\{\widehat{Z}^{mH}_{n}(t), t\geq 0\}$ built on the functional central limit theorem, where the limiting process has a diffusion nature capturing the randomness through the Gaussian process $\{\mathcal{B}(t), t\geq 0\}$ resulting from Proposition \ref{prop_fclt_upsilon_{n}(t)}.

\begin{theorem}\label{thm_FCLT_intensity_process}
	Let the deterministic diffusion convergences hold in Assumptions \ref{assm_base_intensity_FCLT} and \ref{assm_Gamma_FCLT}. Additionally, the assumptions of  Theorem \ref{thm_intensity_LLN} are satisfied. Then $\{\widehat{Z}^{mH}_{n}(t), t\geq 0\}$ converges in distribution to the process $\{\widehat{\mathcal{B}}^{mH}(t), t\geq 0\}$ in $\mathbb{C}([0, \infty), \mathbb{R})$ embedded with Skorokhod $\jt$ topology as $n\ra \infty$,  where the limiting process $\{\widehat{\mathcal{B}}^{mH}(t), t\geq 0 \}$ satisfies the following stochastic integral equation
	\begin{equation}\label{limit_Z^{mH}(t)_intensity_FCLT}
	\widehat{\mathcal{B}}^{mH}(t)= \mathcal B(t)+ \bm{1}( \delta=1-\alpha/2) \left( \int_0^t \widehat{\mu}^{mH}(s) \diff s
    + \int_0^t \widehat{\Gamma}(t,v) \ol{\mu}^{mH} (v)\diff v \right) ,
	\end{equation} 
	where $\{\mathcal B(t), t\geq 0\}$ is a Gaussian process with mean zero and covariance (\ref{limit_B(t)_covariance}). 
\end{theorem}

Knowing the weak convergence of the diffusion scaled intensity process in Theorem \ref{thm_FCLT_intensity_process}, we derive the functional central limit theorem for the marked Hawkes process in next Theorem \ref{thm_FCLT_N^{mH}(t)}, which can be obtained by the weak convergence of the diffusion scaled  process $\{\widehat{N}^{mH}(t), t\geq 0\}$.

\begin{theorem}\label{thm_FCLT_N^{mH}(t)}
	 Suppose the assumptions of Theorem \ref{thm_FCLT_intensity_process} are satisfied. Additionally, the deterministic convergence in Assumption \ref{assm_pi_FCLT} hold on $\mathcal{E}$.
     Then  $\{ \widehat{N}^{mH}_{n}(t), t\geq 0\}$ converges in distribution to $\bm{1}(\delta=1-\alpha/2)\{\widehat{N}^{mH}(t), t\geq 0\}$ in $\md ( [0,\infty), \mathbb{R})$ embedded with Skorokhod $\jt$ topology, where the limiting process is a decomposition of two Gaussian processes $\{\mathcal W(.), t\geq 0\}$ and 
	 $\{\widehat{\mathcal{B}}^{mH}(t), t\geq 0\}$ given by
	 \begin{equation*}
	 \widehat{N}^{mH}(t)= 
	  \mathcal W(t)+\int_0^t \int_{\mathcal E} \ol{\pi}_{s}(\diff z) \diff \widehat{\mathcal{B}}^{mH}(s)+ \int_0^t \int_{\mathcal E} \widehat{\pi}_{s}(\diff z) C^{mH}(s) \diff s, 
	 \end{equation*}
	 where $\{\mathcal W(t), t\geq 0\}$ is a Gaussian process with mean zero and time-dependent variance $\int_0^t \int_{\mathcal E} C^{mH}(s) \ol{\pi}_{s} ({\diff z}) \diff s$ and $\{\widehat{\mathcal{B}}^{mH}(t), t\geq 0\}$ is a diffusion process satisfying the stochastic integral equation (\ref{limit_Z^{mH}(t)_intensity_FCLT}).
\end{theorem}


\begin{remark}\label{remark_fclt_i.i.d}
	In particular, when the mark sequence is i.i.d. in view of Remark \ref{remark_iid of mark sequence}, the resulting process under the multi-scaling regime is accumulation of two Gaussian processes $\{\mathcal W(t), t\geq 0\}$ and $\{\mathcal B(t), t\geq 0\}$, and admits the following form
	\begin{equation}
	\widehat{N}^{mH}(t)=\mathcal W\left (t \right)+\widehat{\mathcal{B}}^{mH}(t), t\geq 0,
	\end{equation}
	where intensity of the Gaussian noise $\{\mathcal W(t), t\geq 0\}$ is given by $\frac{\mu^{mH}}{1-||\Psi||_{\mathbb{L}^{1}}}t$, provided the deterministic convergence of Remark \ref{remark_LLN_i.i.d_mark} holds in $\md([0, \infty), \mathbb{R})$, which is established by \cite{horst2021functional}.

\end{remark}

\begin{remark}\label{remark_fclt_simple Hawkes}
	If we consider the simple Hawkes process without considering the mark sequence analogous to Remark \ref{remark_LLN_simple Hawkes}, then the limiting process turns out to be a single Gaussian process, which is given by 
\begin{equation*}
\widehat{N}^{mH}(t)=  W \left( \frac{\mu^{mH}}{(1-||\Phi||_{\mathbb{L}^{1}})^3}t \right), \q t\geq 0.
\end{equation*}
where $\{ W(t), t\geq 0\}$ is a Wiener process. This particular form the limiting process coincided with the weak limit of \cite{bacry2013some} under suitable assumptions.
\end{remark}

Similar to marked Hawkes process,  we derive the functional central limit theorems for shot noise process under multi-scaling regime, resulting from the Gaussian approximations of cumulative intensity process by virtue of Theorems \ref{thm_FCLT_intensity_process} and \ref{thm_FCLT_N^{mH}(t)}. In addition, we assume the following deterministic convergences. 


\begin{assumption}\label{assm_theta_FCLT}
Given $\ol{\theta}(.,.)$ in Assumption \ref{assm_theta_FLLN}, there exists a locally integrable function $\widehat{\theta}(.,.)$ on $\mathbb{R}_{+}\times \mathcal E$ such that 
\begin{equation}
   \sup_{0\leq t \leq T} \esp_{z\in \mathcal{E}} \left| n^{1-\delta} \left( \theta_{n}(t,z)- \ol{\theta}(t,z) \right) - \widehat{\theta}(t,z)
   \right| \ra  0, \q \xt{as} \q n\ra \infty.
\end{equation}
\end{assumption}
\begin{theorem}\label{thm_FCLT_Y_n(t)_process}
Suppose the assumptions of Theorem \ref{thm_LLN_Y^{mH}(t)} are satisfied. In addition, the convergences of Assumptions \ref{assm_base_intensity_FCLT}, \ref{assm_pi_FCLT}, \ref{assm_theta_FCLT} hold. Then the diffusion scaled shot noise process $\{\widehat{Y}^{mH}_{n}(t), t\geq 0\}$ converges in distribution to the process $\bm{1}(\delta=1-\alpha/2)\{\widehat{Y}^{mH}(t), t\geq 0\}$  in $\md([0, \infty), \mathbb{R})$ as $n\ra \infty$, where 
\begin{align*}
\widehat{Y}^{mH}(t)=&\mathcal W_{\ol{\theta}}(t) + \int_0^t \int_{\mathcal{E}} \ol{\theta}(t-s,z) \ol{\pi}_{s}(\diff z) \diff \widehat{\mathcal{B}}^{mH}(s)
+\int_0^t \int_{\mathcal{E}} \widehat{\theta}(t-s,z) \ol{\pi}_{s}(\diff z) C^{mH}(s) \diff s\\
&+ \int_0^t \int_{\mathcal{E}} \ol{\theta}(t-s,z) \widehat{\pi}_{s}(\diff z)  C^{mH}(s) \diff s,
\end{align*}
where $\{\mathcal W_{\ol{\theta}}(t), t\geq 0\}$ is the Gaussian noise with variance $\int_0^t \int_{\mathcal E} \ol{\theta}^{2}(t-s,z) C^{mH}(s) \ol{\pi}_{s}(\diff z) \diff s $ and $\{\widehat{B}^{mH}(t), t\geq 0\}$ is the diffusion process represented by (\ref{limit_Z^{mH}(t)_intensity_FCLT}).
\end{theorem}



\section{Proofs of  Main Limit Theorems}\label{section_proofs_main_thoerems}
In this section, we provide the detailed proof of the main results including functional law of large numbers and functional central limit theorem for the marked Hawkes processes and the associated shot noise processes in Section \ref{section_FLLN_proofs} and \ref{section_FCLT_proofs} respectively, under the given hypotheses, which are relied on the results derived in Section \ref{section_proofs of auxiliary results}. 

\subsection{Proofs of Auxiliary Results}\label{section_proofs of auxiliary results}
This section provides the proofs of Lemmas \ref{lemma_bound of E lambda^{mH}(t)}, \ref{lemma_conv_phi_wrt_measure}, \ref{lemma_intensity_LLN_expected}, 
\ref{lemma_tau_convergence_tau_infty}, 
\ref{lemma_martingale_upsilon_{n}(t)}, Propositions 
\ref{prop_upsilon_tilde_upsilon_L_2}, 
\ref{prop_FLLN_Upsilon_{n}(t)} in detail, which plays essential roles to prove the main results in following Sections \ref{section_FLLN_proofs} and \ref{section_FCLT_proofs}.

\begin{proof}[Proof of Lemma \ref{lemma_bound of E lambda^{mH}(t)}]
Fix $T > 0$. It follows from (\ref{def_expected_intensity_w.r.t gamma}) that 
	\begin{align}\label{proof_lambda_bound_1}
	\sup_{0\leq t\leq T} \left| \E (\lambda ^{mH}(t)) \right|
	 &\leq \sup_{0\leq t \leq T} |\mu^{mH}(t)| + \sup_{0\leq t \leq T} 
	 \left| \int_0^t \Gamma (t,s) \mu^{mH}(s) \diff s\right|
     \notag \\
	 & \leq ||\mu^{mH}||_{T} \left(1+  \sup_{0\leq t \leq T} 
	  \int_{0}^{t} \left|\Gamma (t,s)\right|  \diff s\right).
    \end{align}
The subcritical condition (\ref{psi_subcritical_condition}) controls the future influence of an event occurred at time $s$ and therefore yields
\begin{equation*}
    \sup_{s\geq 0} \int_{s}^{\infty} |\Gamma(t,s) | \diff t < \infty.
\end{equation*}
However, in order to obtain the finiteness of $\E (\lambda ^{mH}(t)$ uniformly on compact sets, we require 
\begin{equation}\label{gamma_subcritical_condition}
    \sup_{0\leq t \leq T} \int_{0}^{t} |\Gamma(t,s) |\diff s <\infty.
\end{equation}
In particular, for convolution type kernel entity $\psi(t,s):=\psi(t-s)$, the condition (\ref{psi_subcritical_condition}) directly implies 
\begin{equation*}
       \sup_{0\leq t \leq T} \int_{0}^{t} |\psi(t,s) |\diff s = \sup_{0\leq t \leq T} \int_{0}^{t} |\psi(s)| \diff s< \int_{0}^{\infty} |\psi(s)| \diff s< 1,  
\end{equation*}
and hence (\ref{gamma_subcritical_condition}). But in general for the non-convolution kernel $\psi(t,s)$, we additionally assume 
\begin{equation*}
    \rho_{T}:= \sup_{0\leq t \leq T} \int_{0}^{t} |\psi(t,s) |\diff s <1.
\end{equation*}
By (\ref{def_Gamma(t,s)}), we obtain 
\begin{align*}
    \sup_{0\leq t \leq T} \int_{0}^{t} |\Gamma(t,s) |\diff s 
    &\leq \sup_{0\leq t \leq T} \int_{0}^{t} |\psi(t,s) |\diff s + \sup_{0\leq t \leq T} \int_{0}^{t} \left|\Gamma(t,u) \right| \left| 
    \int_{0}^{u} \psi(u,s) \diff s\right| \diff u \\
    &\leq \rho_{T} \left( 1+ \sup_{0\leq t \leq T} \int_{0}^{t} |\Gamma(t,u) |\diff u \right).
\end{align*} 
It implies  
\begin{equation}\label{bound_1}
    \sup_{0\leq t \leq T} \int_{0}^{t} |\Gamma(t,s) |\diff s 
    \leq \frac{\rho_{T}}{1-\rho_{T}} ,
\end{equation} 
and consequently concludes the proof using the relation (\ref{proof_lambda_bound_1}) by virtue of Assumption \ref{assm_initial_intensity_bound} and (\ref{bound_1}).
\end{proof}

\begin{proof}[Proof of Lemma \ref{lemma_conv_phi_wrt_measure}]
    Recall that $\ol{\phi}_n(s-v,z)=0$ for $s<v$. Therefore $\int_0^t \ol{\phi}_{n}(s-v,z) \diff s\equiv \int_{v}^{t} \ol{\phi}_{n}(s-v,z) \diff s$, $t\geq v\geq 0$.
    Consider the difference
    \begin{align*}
        \bigg|\int_0^t \int_{\mathcal{E}} \ol{\phi}_{n}(s-v,z) \pi_{v}^{n}(\diff z) \diff s- \int_{\mathcal{E}}\Phi(z) \ol{\pi}_{v}(\diff z)  \bigg|
        &=  \bigg|\int_{n^{\alpha}v} ^{n^{\alpha}t} \int_{\mathcal{E}} \phi(s-n^{\alpha}v,n^{\beta}z) \pi^{n}_{v}(\diff z) \diff s- \int_{\mathcal{E}}\Phi(z) \ol{\pi}_{v}(\diff z)  \bigg|\\
        &= \bigg|\int_0^{n^{\alpha}(t-v)} \int_{\mathcal{E}} \phi(s, n^{\beta} z) \pi^{n}_{v}(\diff z) \diff s- \int_{\mathcal{E}}\Phi(z) \ol{\pi}_{v}(\diff z)  \bigg|\\
       & \leq \bigg| \int_{0}^{n^{\alpha}(t-v)} \int_{\mathcal{E}} \phi(s, n^{\beta}z) \pi^{n}_{v}(\diff z) \diff s- \int_{\mathcal{E}} 
       \bar{\Phi}(n^{\beta}z) \pi^{n}_{v}(\diff z) \bigg|\\
       &~~~~~~~+ \bigg|\int_{\mathcal{E}} 
       \bar{\Phi}(n^{\beta}z) \pi^{n}_{v}(\diff z)- \int_{\mathcal{E}}\Phi(z) \pi^{n}_{v}(\diff z)  \bigg|\\
       &~~~~~~~+\bigg|\int_{\mathcal{E}} 
       \Phi(z) \pi^{n}_{v}(\diff z)- \int_{\mathcal{E}}\Phi(z) \ol{\pi}_{v}(\diff z)  \bigg|.
    \end{align*}
    Applying Fubini's theorem, we obtain 
    \begin{align*}
    \bigg|\int_0^t \int_{\mathcal{E}} \ol{\phi}_{n}(s-v,z) \pi_{v}^{n}(\diff z) \diff s- \int_{\mathcal{E}}\Phi(z) \ol{\pi}_{v}(\diff z)  \bigg|
    \leq & \esp_{z\in \mathcal{E}}
     \left| \int_{0} ^{n^{\alpha}(t-v)}\phi(s, n^{\beta}z) \diff s- \bar{\Phi}(n^{\beta}z) \right|  \\
     &~~~~+ \esp_{z\in \mathcal{E}} \left| \bar{\Phi}(n^{\beta}z) - \Phi(z)\right|  \\
     &~~~~+ \esp_{z\in \mathcal{E}} |\Phi(z)|
     \int_{\mathcal{E}}   
   \left|   \pi^{n}_{v}(\diff z) - \ol{\pi}_{v}(\diff z) \right| ,
    \end{align*}
    and it goes to zero uniformly on $[\delta , T]$ as $n\ra \infty$ for every $T> \delta>0$ by virtue of (\ref{def_Phi_z_limit_phi}), (\ref{assm_phi_convergence_z_scaling}) and Assumption \ref{assm_pi_FLLN}. The proof is complete.
\end{proof}

\begin{proof}[Proof of Lemma \ref{lemma_intensity_LLN_expected}]
    Consider the difference
	\begin{align}\label{inequality_flln_intensity_1}
	& \E \left( \int_0^t \lambda_{n}^{mH}(s) \diff s\right) - \int_0^{t} C^{mH}(s) \diff s \notag \\
	= &  \int_0^{t} \mu_{n}^{mH} (s) \diff s  - \int_0^t \ol{\mu}^{mH}(s) \diff s 
	+  \int_0^{t} 	\int_0^s \ol{\Gamma}_{n}(s,v)\mu_{n}^{mH} (v) \diff v \diff s - \int_0^t \frac{\Psi(s)}{1-\Psi(s)} \ol{\mu}^{mH}(s)\diff s  \notag \\ 
    = & \int_0^{t} \mu_{n}^{mH} (s) \diff s  - \int_0^t \ol{\mu}^{mH}(s) \diff s 
    +\int_0^{t}\int_0^s \ol{\Gamma}_{n}(s,v) \left(\mu_{n}^{mH} (v) - \ol{\mu}^{mH}(v) \right) \diff v \diff s \notag \\  
    & ~~~~~~ + \int_0^t \int_0^s \ol{\Gamma}_{n}(s,v) \ol{\mu}^{mH}(v)  \diff v \diff s-\int_0^t \frac{\Psi(s)}{1-\Psi(s)} \ol{\mu}^{mH}(s)\diff s \notag \\
     = &  \int_0^{t} \mu_{n}^{mH} (s) \diff s  - \int_0^t \ol{\mu}^{mH}(s) \diff s 
    +\int_0^{t}\int_0^s \ol{\Gamma}_{n}(s,v) \left(\mu_{n}^{mH} (v) - \ol{\mu}^{mH}(v) \right) \diff v \diff s \notag \\ 
    & ~~~~~~~~ + \int_0^t \left( \int_{v}^{t} \ol{\Gamma}_{n}(s,v) \diff s - \frac{\Psi(v}{1-\Psi(v)} \right) \ol{\mu}^{mH}(v)   \diff v.
	\end{align} 
    	To prove the assertion, it suffices to show that the each difference term in (\ref{inequality_flln_intensity_1}) converges to zero on every compact sets $[0,T]$ for all $T>0$ as $n\ra \infty$. The convergence of the first part  of (\ref{inequality_flln_intensity_1}) directly follows from Assumption \ref{assm_base_intensity_FLLN}.

      For the second part of (\ref{inequality_flln_intensity_1}), Fubini's theorem results in
        \begin{align}
            &\sup_{0\leq t \leq T} \left|\int_0^{t}\int_0^s \ol{\Gamma}_{n}(s,v) \left(\mu_{n}^{mH} (v) - \ol{\mu}^{mH}(v) \right) \diff v \diff s \right|\notag \\&
             =   \sup_{0\leq t \leq T}  \left|\int_0^{t}\int_v^t \ol{\Gamma}_{n}(s,v) \left(\mu_{n}^{mH} (v) - \ol{\mu}^{mH}(v) \right) \diff s \diff v \right| \notag \\
            & ~~	\leq  \sup_{\substack{0\leq v \leq t \leq T }}  \left| \int_0^t \ol{\Gamma}_{n} (s, v)  \diff s \right|  \times   \sup_{0\leq t \leq T} \left| \int_0^t  (\mu_{n}^{mH}(v) -\ol{\mu}^{mH}(v)) \diff v  \right| ,
        \end{align}
    which vanishes to zero by Assumption \ref{assm_base_intensity_FLLN} and (\ref{assm_gamma_n_bound}) as $n\ra \infty$.

Similarly, for the third part of (\ref{inequality_flln_intensity_1}), using Fubini's theorem we obtain
\begin{align*}
  & \sup_{0\leq t \leq T} \left|\int_0^t \int_0^s \ol{\Gamma}_{n}(s,v) \ol{\mu}^{mH}(v)  \diff v \diff s-\int_0^t \frac{\Psi(s)}{1-\Psi(s)} \ol{\mu}^{mH}(s)\diff s \right|\\
  = &\sup_{0\leq t \leq T} \left| \int_0^t \left( \int_{v}^{t} \ol{\Gamma}_{n}(s,v) \diff s - \frac{\Psi(v)}{1-\Psi(v)} \right) \ol{\mu}^{mH}(v)   \diff v \right|  \\
  &\leq \sup_{0\leq t \leq T} \left| \int_0^{t} \ol{ \mu} ^{mH}(v)\diff v\right| \times \sup_{\substack{0\leq v \leq t \leq T}}
  \left| \int_{0}^{t} \ol{\Gamma}_{n}(s,v) \diff s - \frac{\Psi(v)}{1-\Psi(v)} \right|,
\end{align*}
which converges to zero uniformly 
on $[\delta , T]$ by Assumption \ref{assm_base_intensity_FLLN} provided $t-v\geq \delta$ for $\delta>0$. The convergence on $[0,\delta]$ follows from (\ref{assm_gamma_LLN}) and (\ref{assm_gamma_n_bound}). The proof is complete.
\end{proof}

\begin{proof}[Proof of Lemma \ref{lemma_tau_convergence_tau_infty}]
By definition (\ref{def_tau_t(v,z)_scaled_new}) of $\mathcal{T}_{n^{\alpha}t} (v,z)$ to prove assertion (\ref{lemma_tau_limit_bound_1}), it suffices to show that 
\begin{enumerate}
\item  \label{tau_bound_first_part}
   $ \sup_{n\geq 1} \sup_{0\leq v \leq t \leq T} \esp_{z\in \mathcal{E}}  \left| \int_v^t  \ol{\phi}_{n}(w-v,z)  \diff w \right|<\infty$.
\item \label{tau_bound_second_part}
    $\sup_{n\geq 1} \sup_{0\leq v \leq t \leq T} \esp_{z\in \mathcal{E}}  \left| \int_v^t \int_w^t \ol{\phi}_{n}(w-v,z)  \ol{\Gamma}_{n}(s,w)  \diff s \diff w \right| <\infty$.
\end{enumerate}
It is straightforward to verify that (\ref{tau_bound_first_part}) and (\ref{tau_bound_first_part}) hold by reason of (\ref{assm_phi_bound_LLN}) and (\ref{assm_gamma_n_bound}).


     In order to prove the second part (\ref{lemma_tau_conv_2}), we define 
    \begin{align*}
        A_{n}(t,w) =& 1 +\int_w^t \ol{\Gamma}_{n}(s,w)  \diff s, \q t \geq w \geq 0\\
        a(w)= & \frac{1}{1- \Psi(w)},
    \end{align*}
    and by (\ref{assm_gamma_LLN}), for any $\eta>0$ and $T>0$ we get
    \begin{equation}\label{convergence_A_n}
        \sup_{\substack{\eta \leq w \leq t \leq T\\t-w \geq \eta}} \left|A_{n}(t,w) - a(w) \right|\ra 0 , \q \text{as} \q n \ra \infty. 
    \end{equation}
    Fix $\delta>0$. For $t-v\geq \delta$, consider the difference
    \begin{align}\label{proof_tau_n_conv_inequality_1}
    \mathcal{T}_{n^{\alpha}t}(v,z) - \mathcal{T}_{\infty}(v,z)    
       =&  \int_v^t \ol{\phi}_{n} (w-v,z) A_n(t,w) \diff w -  \Phi(z) a(v)  \notag \\
      &\leq  \int_{v}^{t} \ol{\phi}_{n}(w-v,z) 
      \bigg(A_{n}(t,w) - a(w) \bigg)  \diff w    
      + \int_{v}^{t} \ol{\phi}_{n}(w-v,z) 
      \left(a(w) -a(v) \right)  \diff w \notag\\
      & ~~~~+ a(v) \left(\int_{v}^{t} \ol{\phi}_{n}(w-v,z)\diff w- \Phi(z)\right).
      \end{align}

      For the first part, fix $\eta\in (0,\delta)$, split the first integral part of (\ref{proof_tau_n_conv_inequality_1}) into
      \begin{align*}
          \mathcal{T}_{1}^{n,t}(v,z) =& \int_{v}^{t-\eta} \ol{\phi}_{n}(w-v,z) \left(A_n(t,w) -a(w) \right)  \diff w ,\\
          \mathcal{T}_{2}^{n,t}(v,z) =& \int_{t-\eta}^{t} \ol{\phi}_{n}(w-v,z) \left(A_{n}(t,w) -a(w) \right)  \diff w .
      \end{align*}
      Now, for $w\in [v,t-\eta]$, the convergence  (\ref{convergence_A_n}) of $A_{n}(t,w)$ together with (\ref{assm_phi_bound_LLN}) yields that $\mathcal{T}_{1}^{n,t}(v,z)$ converges to zero uniformly on $t$ and $z$. For the latter entity $\mathcal{T}_{2}^{n,t}(v,z)$, note that $w-v \in (t-\eta-v, t-v]$ and $t-v\geq \delta$ and $\eta\in (0,\delta)$ together imply $t-\eta-v\geq \delta-\eta>0$. As a consequence, the interval $w\in (t-\eta, t]$ in $\mathcal{T}_{2}^{n,t}(v,z)$ is away from zero, but the concentration of $\ol{\phi}_{n}(.,.)$ from Assumption \ref{assm_phi_conv_LLN} lies near zero only, which results 
      \begin{equation*}
          \lim_{n\ra \infty} \sup_{\substack{0\leq v \leq t\leq T\\t-v\geq \delta}} \esp_{z\in \mathcal{E}} \int_{t-\eta}^{t} \left| \ol{\phi}_{n}(w-v,z) \right|\diff w =0 .
      \end{equation*}
      Hence $\mathcal{T}_{2}^{n,t}(v,z)$ vanishes to zero uniformly. 

    For the second integral part of (\ref{proof_tau_n_conv_inequality_1}), we use the concentration property of $\ol{\phi}_{n}(w-v,z)$ at the diagonal $w=v$ inside the limiting regime. Due to this mass concentration around $w=v$, the entity $a(w)$ eventually behaves as $a(v)$ in the limiting regime and hence, without loss of generality, we can say that 
     \begin{equation*}
     \sup_{\substack{0\leq v \leq t\leq T\\t-v\geq \delta}}
     \esp_{z\in \mathcal{E}}
         \left|\int_{v}^{t} \ol{\phi}_{n}(w-v,z) 
      \left(a(w) -a(v) \right)  \diff w\right| \ra 0 , \; \text{as} \; n\ra \infty.
     \end{equation*}

     Knowing $t-v\geq \delta$, it is easy to verify that the third part of (\ref{proof_tau_n_conv_inequality_1}) converges to zero uniformly on $0\leq v\leq t\leq T$ and $z$ by reason of Assumption \ref{assm_phi_conv_LLN}.

     The proof is complete.      
\end{proof}

\begin{proof}[Proof of Proposition \ref{prop_upsilon_tilde_upsilon_L_2}]
   For every $f\in \mathbb{C}_{b} (\mathcal{E}, \mathbb{R})$, define 
   \begin{equation*}
       \widehat{R}_{n,t}(f)= \frac{1}{n^{\alpha/2}} \int_0^{n^{\alpha}t} \int_{\mathcal{E}} \int_{0}^{\lambda^{mH}(v-)}f(z) \bm{1}(v\leq n^{\alpha}t) \left( \mathcal{T}_{n^{\alpha}t}(v,z) - \mathcal{T}_{\infty}(v,z)\right) M^{P}(\diff v, \diff z, \diff u)
   \end{equation*}
   such that $\widehat{R}_{n,t}(f)= \frac{1}{n^{\alpha/2}} \Upsilon_{n,t}(f) - \frac{1}{n^{\alpha/2}} \widetilde{\Upsilon}_{n,t}(f)$.
   To prove the assertion it suffices to prove that for every $f\in \mathbb{C}_{b} (\mathcal{E}, \mathbb{R})$, $T>0$ and $\epsilon>0$,
   \begin{equation*}
       \lim_{n\ra \infty} \mathbb{P} \left( \sup_{0\leq t \leq T} \left|\widehat{R}_{n,t}(f) \right| > \epsilon\right)=0.
   \end{equation*}
   

   Define 
   \begin{equation*}
       C_{n}([t_1 , t_2]) = \sup_{t_1 \leq v \leq t \leq t_2} \esp_{z\in \mathcal{E}} \left| \mathcal{T}_{n^{\alpha}t}(n^{\alpha}v,z) - \mathcal{T}_{\infty}(n^{\alpha}v,z)\right|, \; t_1, t_2\geq 0.
   \end{equation*}
   Now by change of variables we rewrite  
   \begin{align*}
       &\sup_{0\leq t \leq T} |\widehat{R}_{n,t}(f)|\\
       &= \sup_{0\leq t \leq T}\left| \frac{1}{n^{\alpha/2}} \int_0^{t} \int_{\mathcal{E}} \int_{0}^{\lambda^{mH}_{n}(v-)}f(z) \bm{1}(n^{\alpha}v\leq n^{\alpha}t) \left( \mathcal{T}_{n^{\alpha}t}(n^{\alpha}v,z) - \mathcal{T}_{\infty}(n^{\alpha}v,z)\right) M^{P}(n^{\alpha}\diff v, \diff z, \diff u) \right|\\
       &~~\leq ||f||_{\mathcal{E}} C_{n}([0,T]) 
       \sup_{0\leq t \leq T} \left| \frac{1}{n^{\alpha/2}}  \int_0^{t} \int_{\mathcal{E}} \int_{0}^{\lambda^{mH}_{n}(v-)}
       M^{P}(n^{\alpha}\diff v, \diff z, \diff u)\right|.
   \end{align*}
   By given hypothesis, $ \{\frac{1}{n^{\alpha/2}}  \int_0^{t} \int_{\mathcal{E}} \int_{0}^{\lambda^{mH}_{n}(v-)}
       M^{P}(n^{\alpha}\diff v, \diff z, \diff u), t\geq 0 \}$ is an $\mathscr{F}_{n}$-martingale. Therefore, it follows from Doob's $\mathbb{L}^2$-inequality that 
       \begin{align*}
        \E \left( \sup_{0\leq t \leq T} \left|\frac{1}{n^{\alpha/2}}  \int_0^{t} \int_{\mathcal{E}} \int_{0}^{\lambda^{mH}_{n}(v-)}
       M^{P}(n^{\alpha}\diff v, \diff z, \diff u) \right|^{2}\right) 
       &\leq 
       4 \E \left( \frac{1}{n^{\alpha}} \int_0^{T} \int_{\mathcal{E}} \lambda^{mH}_{n}(v)\pi_{n^{\alpha}v}(\diff z) n^{\alpha} \diff v  \right) \\ 
       &= 4 \E \left( \int_0^{T}  \lambda^{mH}_{n}(v) \diff v\right).
       \end{align*}
       It leads to for every $\epsilon>0$,
       \begin{align}\label{inequality_fclt_intensity_3}
     \mathbb{P} \left( \sup_{0\leq t \leq T} \left|\widehat{R}_{n,t}(f) \right| > \epsilon\right)
      &\leq \frac{1}{\epsilon^2} C_{n}([0,T])^{2} ||f||^{2}_{\mathcal{E}} \E \left( \sup_{0\leq t \leq T} \left|\frac{1}{n^{\alpha/2}}   \int_0^{t} \int_{\mathcal{E}} \int_{0}^{\lambda^{mH}_{n}(v-)}
       M^{P}(n^{\alpha}\diff v, \diff z, \diff u) \right|^{2}\right) \notag \\
       & \leq \frac{1}{\epsilon^2} C_{n}([0,T])^{2} ||f||^{2}_{\mathcal{E}}.  4 \E \left( \int_0^{T}  \lambda^{mH}_{n}(v) \diff v\right) \notag \\
       & \leq \frac{8}{\epsilon^2}  ||f||^{2}_{\mathcal{E}}
       \Bigg( C_{n}([0,\delta])^{2} \E \left( \int_0^{\delta}  \lambda^{mH}_{n}(v) \diff v\right)
       +C_{n}([\delta,T])^{2} \E \left( \int_{\delta}^{T}  \lambda^{mH}_{n}(v) \diff v\right)       \Bigg).
       \end{align} 
       In the end, taking $\delta \downarrow 0$, Lemma \ref{lemma_intensity_LLN_expected} and the convergence   (\ref{lemma_tau_conv_2}) of Lemma \ref{lemma_tau_convergence_tau_infty} respectively result that the both term in (\ref{inequality_fclt_intensity_3}) vanishes to zero as $n\ra \infty$. 
\end{proof}

	\begin{proof}[Proof of Lemma \ref{lemma_martingale_upsilon_{n}(t)}]
	Recall that the compensated Poisson measure $M^{P} (\diff v, \diff z, \diff u)$ is a martingale measure with the intensity measure $\diff v \pi_{v}(\diff z) \diff u$. To prove the assertion, by virtue of Theorem T8($\beta$) in \cite{bremaud1981} it suffices to show that for each $n\in \mathbb{N}$,
	\begin{enumerate}[(i)]
		\item \label{condition_1_martingale}
		$\mathcal T_{\infty}(v,z)$ is locally bounded on $\mathbb{R}_{+}\times \mathcal E$.
		\item \label{condition_2_martingale}
		$\E \int_0^{n^{\alpha}t} \int_{\mathcal E}  \left| f(z)  \mathcal T_{\infty}(v,z) \lambda^{mH}(v) \right|  \pi_{v}(\diff z) \diff v < \infty$, for all $t\geq 0$.
	\end{enumerate}
	By given hypotheses, the condition (\ref{condition_1_martingale}) directly follows from 
    Assumption \ref{assm_phi_conv_LLN} and (\ref{assm_gamma_LLN}). Similarly, Lemma \ref{lemma_bound of E lambda^{mH}(t)} together with Assumption \ref{assm_phi_conv_LLN} and (\ref{assm_gamma_LLN}) imply the second condition (\ref{condition_2_martingale}).

   The quadratic variation of the martingale $\widetilde{\Upsilon}_{n,t}(f)$ can be calculated by the following relation in virtue of Lemma 5.7 in \cite{vandervaart1996},
	\begin{align*}
	\langle \widetilde{\Upsilon}_{n,.}(f), \widetilde{\Upsilon}_{n,.}(g) \rangle (t)
    &= 
	\int_0^{n^{\alpha}t} \int_{\mathcal E} \int_0^{\lambda^{mH}(v-)} f(z) g(z) \mathcal{T}^{2}_{\infty}(v,z) \langle M^{P}(\diff v, \diff z, \diff u) \rangle \\
	&=	 \int_0^{n^{\alpha}t} \int_{\mathcal E}  f(z) g(z) \mathcal{T}_{\infty}^{2}(v,z) \lambda^{mH}(v)  \pi_v(\diff z) \diff v.
	\end{align*}
	The proof is complete.
\end{proof}	


\begin{proof}[Proof of Proposition \ref{prop_FLLN_Upsilon_{n}(t)}]
It is already proved in Lemma \ref{lemma_martingale_upsilon_{n}(t)} that	$\{\frac{1}{n^{\alpha/2+\eta}} \widetilde{\Upsilon}_{n}(t), t\geq 0\}$ is an $\mathscr{F}_{n}$-square integrable martingale with predictable quadratic variation process
	\begin{equation*}
	\frac{1}{n^{\alpha+2\eta}} \langle \widetilde{\Upsilon} _{n} \rangle (t)
	= \frac{1}{n^{\alpha+2\eta}}  \int_0^{n^{\alpha}t} \int_{\mathcal E}   \mathcal{T}_{\infty}^{2}(v,z) \lambda^{mH}(v)  \pi_v(\diff z) \diff v. 
	\end{equation*}
	To prove the assertion, it is enough to show that the entity $\frac{1}{n^{\alpha+2\eta}} \langle \widetilde{\Upsilon}_{n} \rangle (t)$ converges to zero in probability for every $t\geq 0$ as $n\ra\infty$.  On that account, for every $\theta>0$, using Markov's inequality we obtain 
	\begin{align}\label{inequality_upsilon_LLN}
     &	\prob \left( \left|\frac{1}{n^{\alpha+2\eta}}  \int_0^{n^{\alpha}t} \int_{\mathcal E}   \mathcal{T}_{\infty}^{2}(v,z) \lambda^{mH}(v)  \pi_v(\diff z) \diff v \right|>\theta \right) \notag \\
     & \leq 
	\frac{1}{\theta} \frac{1}{n^{\alpha+2\eta}}  \int_0^{n^{\alpha}t} \int_{\mathcal E}   \mathcal{T}_{\infty}^{2}(v,z) \E \lambda^{mH}(v)  \pi_v(\diff z) \diff v.
	\end{align}
	The boundedness of $\E \lambda^{mH}(t)$ by Lemma \ref{lemma_bound of E lambda^{mH}(t)} and of  $ \mathcal{T}_{\infty}^{2}(v,z)$ together imply that the integral entity on right hand side of (\ref{inequality_upsilon_LLN}) is bounded by $\mathcal{O}(1/n^{2\eta})$, which vanishes to zero as $n\ra \infty$ for all $\eta>0$. In the end, in accordance with Lenglart and Rebolledo inequality in view of Theorem 1.9.3 in \cite{liptser2012theory}, it follows that for every $T>0$ and $\epsilon>0$, $\theta>0$, 
	\begin{align*}
	\prob \left( \sup_{0\leq t \leq T}	\frac{1}{n^{\alpha/2+\eta}} |\widetilde{\Upsilon}_{n}(t)|>\epsilon \right) \leq & \prob 	\left( \frac{1}{n^{\alpha+2\eta}}  \langle  \widetilde{\Upsilon}_{n} \rangle (T) >\theta \right) +\frac{\theta }{\epsilon ^2}\\
	& 	\leq \frac{1}{\theta} \frac{1}{n^{\alpha+2\eta}}  \int_0^{n^{\alpha}T} \int_{\mathcal E}   \mathcal{T}_{\infty}^{2}(v,z) \E \lambda^{mH}(v)  \pi_v(\diff z) \diff v+ \frac{\theta }{\epsilon ^2}.
	\end{align*}
	Taking $n\ra \infty$, the above inequality concludes the requisite convergence. 
\end{proof}

\subsection{Proof of Functional Law of Large Number Theorems}\label{section_FLLN_proofs}
In this section, we particularly focus on the derivation of the functional law of large numbers of the marked Hawkes process $\{N^{mH}_{n}(t), t \geq 0\}$ along with the intensity process $\{\lambda ^{mH}_{n}(t), t\geq 0\}$. In addition, we provide the functional law of large numbers for the shot noise process $\{Y^{mH}_{n}(t), t\geq 0\}$.

\begin{proof}[Proof of Theorem \ref{thm_intensity_LLN}]
	Using definition (\ref{def_Upsilon_{n}(t)}), we write 
	\begin{align} \label{inequality_inetnsity_LLN}
    \int_0^t \lambda _{n}(s)\diff s-\int_0^t C^{mH}(s)\diff s 
    =& \int_0^t \lambda^{mH}_{n}(s)\diff s- \int_0^t \E \lambda ^{mH}_{n}(s) \diff s +  \int_0^t \E \lambda ^{mH}_{n}(s) \diff s
    \notag 
    \\ & ~~~~-\int_0^t C^{mH}_{n}(s) \diff s \notag \\
   =& \frac{1}{n^{\alpha}} \Upsilon_{n}(t)+ \int_0^t \E \lambda ^{mH}_{n}(s) \diff s-\int_0^t C^{mH}(s) \diff s,
\end{align}
where $\Upsilon_{n}(t): = \Upsilon_{n,t}(\bm 1)$. In particular, the convergence of first term of (\ref{inequality_inetnsity_LLN}) is a direct consequences of Propositions  \ref{prop_upsilon_tilde_upsilon_L_2} and  \ref{prop_FLLN_Upsilon_{n}(t)} for $\eta=\alpha/2$. On the other hand, Lemma \ref{lemma_intensity_LLN_expected} yields the convergence of the second part of (\ref{inequality_inetnsity_LLN}) and hence concludes the assertion of the theorem.
\end{proof}

\begin{proof}[Proof of Theorem \ref{thm_LLN_N^{mH}(t)}]
	Given an $\mathscr{B}_{\mathcal E}$-measurable function $f\in \mathbb{C}_{b}(\mathcal{E}, \mathbb{R})$, we define
    \begin{equation}\label{def_N^{mH}_bar_f}
        \overline{N}^{mH}_{t}(f) = \int_0^t \int_{\mathcal{E}} f(z) C^{mH}(s) \ol{\pi}_{s}(\diff z) \diff s.
    \end{equation}
    We consider the difference 
	\begin{align} \label{inequality_flln_N(t)_1}
	&\frac{1}{n^{\alpha}} N^{mH}_{n,t}(f)- \overline{N}^{mH}_{t}(f)  \notag \\
	=& \frac{1}{n^{\alpha}} \int_0^t \int_{\mathcal E}	f(z) M_{n}^{mH} (\diff s,  \diff z)
    + \int_0^t \int_{\mathcal E} f(z) \left(\lambda^{mH}_{n}(s) -C^{mH}(s) \right) \pi^{n}_{s} ( \diff z)\diff s \notag \\  
	&~~~~ + \int_0^t \int_{\mathcal{E}} f(z) C^{mH}(s) \left( \pi^{n}_{s} ( \diff z) - \ol{\pi}_{s}(\diff z)\right) \diff s,
	\end{align}
    where $M^{mH}_{n}(\diff s, \diff z) = N_{n}^{mH} (\diff s, \diff z) - \lambda^{mH}_{n}(s) n^{\alpha} \diff s \pi^{n}_{s}(\diff z) $.
	For every $\mathscr{B}_{\mathcal E}$-measurable bounded function $f$, the double integral process $\{ \frac{1}{n^{\alpha}} \int_0^t \int_{\mathcal E} f(z) M_{n}^{mH} (\diff s,  \diff z), t\geq 0\}$ is an $\mathscr{F}_{n}$-martingale with quadratic variation process  
	\begin{align*}
	& \left \langle  \frac{1}{n^{\alpha}}  \int_0^. \int_{\mathcal E} f(z) M_{n}^{mH} (\diff s,\diff z), \frac{1}{n^{\alpha}}  \int_0^. \int_{\mathcal E} g(z) M_{n}^{mH} (\diff s, \diff z) \right \rangle (t)\\
	&= \frac{1}{n^{\alpha}} \int_0^t \int_{\mathcal E} f(z)g(z) \lambda^{mH}_{n}(s)  \pi^{n}_{s}(\diff z) \diff s.
	\end{align*}
	In view of Lenglart and Rebolledo inequality (see Theorem 1.9.3 of \cite{liptser2012theory}), for every $T>0$ and $\epsilon>0, \theta>0$, we get 
	\begin{align}\label{proof_inequality_N_Mh_process}
	&  \prob \left(\sup_{0\leq t \leq T} \frac{1}{n^{\alpha}} \left|  \int_0^t \int_{\mathcal E}	f(z) M_{n}^{mH} (\diff s,  \diff z) \right| >\epsilon \right) \notag \\
	&	\leq \prob\left(  \frac{1}{n^{\alpha}}	\int_0^T \int_{\mathcal E} f^2(z) \lambda^{mH}_{n}(s)   \pi^{n}_{s}(\diff z) \diff s>\theta 	\right)+\frac{\theta}{\epsilon^2} \notag \\
	& \leq \frac{1}{\theta} \frac{1}{n^{\alpha}}\int_0^T \int_{\mathcal E} f^2(z) \E \lambda^{mH} _{n} (s) \pi^{n}_{s}(\diff z) \diff s 	+\frac{\theta}{\epsilon^2}	.
	\end{align} 
    Now, in view of Assumption \ref{assm_pi_FLLN} and Theorem \ref{thm_intensity_LLN}, $\int_0^t \int_{\mathcal E} f^2(z) \E \lambda^{mH} _{n} (s) \pi^{n}_{s}(\diff z) \diff s $ converges to $\int_0^t \int_{\mathcal E} f^2(z) C^{mH} (s) \ol{\pi}_{s}(\diff z) \diff s $ uniformly on compact sets. Therefore, the first double integral term of (\ref{proof_inequality_N_Mh_process}) is bounded by $\mathcal{O}(1/n^{\alpha})$, which vanishes to zero as $n\ra \infty$. It essentially shows by reason of Lenglart and Rebolledo inequality that the martingale part of (\ref{inequality_flln_N(t)_1}) converges to zero process in probability uniformly on compact sets as $n\ra \infty$.

    For the second integral part of (\ref{inequality_flln_N(t)_1}), we obtain
    \begin{align*}
     &\prob \left( \sup_{0\leq t \leq T} \left|  \int_0^t \int_{\mathcal E} f(z) \left(\lambda^{mH}_{n}(s) -C^{mH}(s) \right) \pi^{n}_{s}(\diff z) \diff s \right|>\epsilon \right)  \\
     &\leq 
          \prob \left( ||f||_{\mathcal{E}}\sup_{0\leq t \leq T} \left|   \int_0^t  \left(\lambda^{mH}_{n}(s) -C^{mH}(s) \right)  \diff s \right|>\epsilon \right) ,
    \end{align*}
    and the latter probability vanishes to zero as $n\ra \infty$ by virtue of Theorem \ref{thm_intensity_LLN}.

	To conclude the assertion of the theorem, it remains to show that 
	\begin{equation*}
	\lim_{n\ra \infty} \prob \left( \sup_{0\leq t \leq T} \left|  \int_0^t \int_{\mathcal E} f(z) C^{mH}(s) \left( \pi^{n}_{s} ( \diff z) - \ol{\pi}_{s}(\diff z) \right) \diff s \right|>\epsilon \right) =0,
	\end{equation*}
    which immediately follows from Assumption \ref{assm_pi_FLLN}.

	Substituting $f\equiv \bm 1$ in the representation (\ref{def_N^{mH}_bar_f}) and applying the continuous mapping theorem, we conclude the proof.
\end{proof}

\begin{proof}[Proof of Theorem \ref{thm_LLN_Y^{mH}(t)}]

	Consider the difference 
	\begin{align}\label{proof_Y_mH_LLN_1}
	&\frac{1}{n^{\alpha}} Y^{mH}_{n}(t)- \ol {Y}^{mH}(t) \notag \\&
	= 	 \frac{1}{n^{\alpha}} \int_0^t \int_{\mathcal E} \theta_{n}(t- s,z) M_{n}^{mH}( \diff s, \diff z) + 	\int_0^t  \int_{\mathcal E} \theta_{n}(t-s,z) 
    \left (\lambda^{mH}_{n}(s) - C^{mH}(s) \right) \pi^{n}_{s}(\diff z) \diff s \notag \\
	& ~~~~~+\int_0^t C^{mH}(s)    \left( \int_{\mathcal E} \theta_{n}(t-s,z)  \pi^{n}_{s}(\diff z) - \int_{\mathcal E} \ol{\theta}(t-s,z) \ol{\pi}_{s}(\diff z) \right) \diff s .
	\end{align}

     For the first part $\{ \frac{1}{n^{\alpha}} \int_0^t \int_{\mathcal E} \theta_{n}(t- s,z) M_{n}^{mH}( \diff s, \diff z), t\geq 0 \}$ of (\ref{proof_Y_mH_LLN_1}), we apply the martingale argument used in the earlier proof of Theorem \ref{thm_LLN_N^{mH}(t)}. 
     By definition, $\{ \frac{1}{n^{\alpha}} \int_0^t \int_{\mathcal E} \theta_{n}(t- s,z) M_{n}^{mH}( \diff s, \diff z), t\geq 0 \}$ is an $\mathscr{F}_{n}$-square integrable martingale with quadratic variation process 
	\begin{equation*}
	\left \langle  \frac{1}{n^{\alpha}} \int_0^. \int_{\mathcal E} \theta_{n}(t- s,z) M_{n}^{mH}( \diff s, \diff z) \right \rangle (t)= \frac{1}{n^{\alpha}} \int_0^t
	\int_{\mathcal E} \theta_{n}^{2}(t-s,z) \lambda ^{mH}_{n}(s) \pi_{s}^{n}(\diff z) \diff s.
	\end{equation*}
    Now, by Assumptions \ref{assm_pi_FLLN} and \ref{assm_theta_FLLN} and Theorem \ref{thm_intensity_LLN}, we obtain 
    \begin{align*}
    \lim_{n\ra \infty} \prob \left( \sup_{0\leq t \leq T} \left| \int_0^t
	\int_{\mathcal E} \theta_{n}^{2}(t-s,z) \lambda ^{mH}_{n}(s) \pi_{s}^{n}(\diff z) \diff s- \int_0^t
	\int_{\mathcal E} \ol{\theta}^{2}(t-s,z) C^{mH} (s) \ol{\pi}_{s} (\diff z) \diff s\right|
    \right) =0 ,
    \end{align*}
    which concludes that the quadratic variation entity $\big \langle  \frac{1}{n^{\alpha}} \int_0^. \int_{\mathcal E} \theta_{n}(t- s,z) M_{n}^{mH}( \diff s, \diff z) \big \rangle (t)$ converges to zero in probability for all $t\geq 0$ as $n\ra \infty$. Analogous to the previous proof,  this convergence is sufficient to establish that the martingale  $\{\frac{1}{n^{\alpha}} \int_0^t \int_{\mathcal E} \theta_{n}(t- s,z) M_{n}^{mH}( \diff s, \diff z), t\geq 0 \}$ converges to zero process in probability uniformly on compact sets as $n\ra  \infty$ as a consequence of Theorem 1.9.3 of \cite{liptser2012theory}.

For the second part of (\ref{proof_Y_mH_LLN_1}), 
\begin{align*}
&\prob \left( \sup_{0\leq t \leq T} \left| \int_0^t  \int_{\mathcal E} \theta_{n}(t-s,z) 
    \left (\lambda^{mH}_{n}(s) - C^{mH}(s) \right) \pi^{n}_{s}(\diff z) \diff s  \right| > \epsilon  \right)\\
   & \leq \prob \left( \sup_{0\leq t \leq T} \left| \int_0^t   \esp_{z\in \mathcal{E}} |\theta_{n}(t-s,z) |
    \left (\lambda^{mH}_{n}(s) - C^{mH}(s) \right) \diff s  \right| > \epsilon  \right), 
\end{align*}
which converges to zero as $n\ra \infty$ by the given hypothesis of $\theta_{n}(t,z)$ in Assumption \ref{assm_theta_FLLN} and  Theorem \ref{thm_intensity_LLN}.

The third part of (\ref{proof_Y_mH_LLN_1}) leads to 
\begin{align*}
&\prob \left( \sup_{0\leq t \leq T} \left| \int_0^t C^{mH}(s)    \left( \int_{\mathcal E} \theta_{n}(t-s,z)  \pi^{n}_{s}(\diff z) - \int_{\mathcal E} \ol{\theta}(t-s,z) \ol{\pi}_{s}(\diff z) \right) \diff s . \right| > \epsilon  \right)\\
&\leq  \prob \left( \sup_{0\leq t \leq T} \left| \int_0^t C^{mH}(s)   \int_{\mathcal E} \left( \theta_{n}(t-s,z) -\ol{\theta}(t-s,z) \right)   \pi^{n}_{s}(\diff z)  \right| > \frac{\epsilon}{2}  \right) \\
&~~~~~+  \prob \left( \sup_{0\leq t \leq T} \left| \int_0^t C^{mH}(s)   \int_{\mathcal E} \ol{\theta}(t-s,z)  \left( \pi^{n}_{s}(\diff z) - \ol{\pi}_{s}(\diff z) \right)  \right| > \frac{\epsilon}{2}  \right) 
\end{align*}
which converges to zero as $n\ra \infty$ by virtue of Assumption \ref{assm_pi_FLLN} and \ref{assm_theta_FLLN}. Finally, the direct application of the continuous mapping theorem concludes the convergence of (\ref{proof_Y_mH_LLN_1}).
\end{proof}

\subsection{Proof of Functional Central Limit Theorems}\label{section_FCLT_proofs}
In this section, we prove the functional central limit theorems for the processes of interests under the multi-scaling regime. Before proceeding to the main proofs, we prove the weak convergence of $\{\widehat{\Upsilon}_{n}(f)(t), t\geq 0\}$ in $\md ([0, \infty), \mathbb{R})$ in Proposition \ref{prop_fclt_upsilon_{n}(t)}.

  \begin{proof}[Proof of Proposition \ref{prop_fclt_upsilon_{n}(t)}]
  In order to establish the weak convergence of $\{\widehat{\Upsilon}_{n,t}(f), t\geq 0\}$, we prove the following.
  \begin{enumerate}[(1)]
    \item \label{proof_upsilon_weak_conv_1}
    $\left \{\frac{1}{n^{\alpha/2}}\widetilde{\Upsilon}_{n,t}(f), t\geq 0\right\}$ converges in distribution to a Gaussian process $\{\mathcal{B}_{t}(f)) , t \geq 0\}$ in $\mathbb{D}\left([0, \infty), \mathbb {R} \right)$ as $n\ra \infty$, where 
    \begin{align}
    \mathbb{E} ( \mathcal{B}_{t}(f)) = 0 , \quad 
    \mathbb{E} ( \mathcal{B}_{t}(f) \mathcal{B}_{t}(g) ) 
    = \int_0^t \int_{\mathcal{E}} f(z) g(z) \mathcal{T}^2_{\infty}(v,z) C^{mH}(v) \ol{\pi}_{v}(\diff z) \diff v. 
    \end{align}
    \item \label{proof_upsilon_weak_conv_2}
    For every $f\in \mathbb{C}_{b}(\mathcal{E}, \mathbb{R})$, $T>0$ and $\epsilon>0$,
    \begin{equation}
        \lim_{n\ra \infty} \mathbb{P} \left(\sup_{0 \leq t \leq T} \left| \widehat{\Upsilon}_{n,t}(f) - \frac{1}{n^{\alpha/2}}\widetilde{\Upsilon}_{n,t}(f)
        \right| > \epsilon \right) =0 .
    \end{equation} 
\end{enumerate}
  By view of Theorem 3.1 of \cite{billingsley1999convergence}, (\ref{proof_upsilon_weak_conv_1}) and (\ref{proof_upsilon_weak_conv_2}) together yield the requisite convergence of $\{\widehat{\Upsilon}_{n,t}(f), t\geq 0\}$ in $\mathbb{D}\left([0, \infty), \mathbb {R} \right)$ embedded with Skorokhod $\jt$ topology. Here, the second assertion (\ref{proof_upsilon_weak_conv_2}) directly follows from Proposition \ref{prop_upsilon_tilde_upsilon_L_2}. Therefore, it  only remains to prove the first assertion (\ref{proof_upsilon_weak_conv_1}).

  To prove the first assertion (\ref{proof_upsilon_weak_conv_1}), we begin with the fact that, in accordance with Lemma \ref{lemma_martingale_upsilon_{n}(t)}, $\left \{\frac{1}{n^{\alpha/2}}\widetilde{\Upsilon}_{n,t}(f), t\geq 0\right \}$ is an $\mathscr{F}_{n}$-square-integrable martingale with quadratic covariation entity
	\begin{align}\label{QV_upsilon_n_t(f,g)}
    \left \langle \frac{1}{n^{\alpha/2}}\widetilde{\Upsilon}_{n,.}(f),\frac{1}{n^{\alpha/2}}\widetilde{\Upsilon}_{n,.}(g) \right \rangle (t)
    = &	\frac{1}{n^{\alpha}} \langle \widetilde{\Upsilon}_{n,.}(f), \widetilde{\Upsilon}_{n,.}(g) \rangle (t) \notag \\
    =&
	\frac{1}{n^{\alpha}} \int_0^{n^{\alpha}t} 
    \int_{\mathcal E}  f(z) g(z) \mathcal T_{\infty}^{2}(v,z) \lambda^{mH}(v)   \pi_v(\diff z) \diff v \notag \\
    =& \int_0^t \int_{\mathcal{E}} f(n^{\beta} z) g(n^{\beta} z)\mathcal T_{\infty}^{2}(n^{\alpha} v, n^{\beta} z) \lambda^{mH}_{n}(v) \pi_{v}^{n}(n^{\beta} \diff z) \diff v  .
	\end{align} 
	Next thing is to derive the limit of the predictable covariation process $\{\langle \frac{1}{n^{\alpha/2}} \widetilde{\Upsilon}_{n,.}(f), \frac{1}{n^{\alpha/2}} \widetilde{\Upsilon}_{n,.}(g) \rangle (t), t\geq 0\}$ in $\md([0,\infty), \mathbb{R})$ for given $f,g \in \mathbb{C}_{b}(\mathcal E, \mathbb{R})$. Our aim is to show that 
	\begin{equation*}
	\lim_{n\ra \infty} \prob \bigg(\sup_{0\leq t \leq T} \bigg|  \left \langle \frac{1}{n^{\alpha/2}}\widetilde{\Upsilon}_{n,.}(f),\frac{1}{n^{\alpha/2}}\widetilde{\Upsilon}_{n,.}(g) \right \rangle (t)
    - \int_0^t \int_{\mathcal E} f(z) g(z) \mathcal{T}_{\infty}^2(v,z) C^{mH}(v) \ol{\pi}_{v}(\diff z) \diff v \bigg|>\epsilon \bigg ) =0.
	\end{equation*}
    To prove this, we consider
	\begin{align}\label{inequality_upsilon_fclt}
     &\bigg|  \left \langle \frac{1}{n^{\alpha/2}}\widetilde{\Upsilon}_{n,.}(f),\frac{1}{n^{\alpha/2}}\widetilde{\Upsilon}_{n,.}(g) \right \rangle (t)
      - \int_0^t \int_{\mathcal E} f(z) g(z)\mathcal T_{\infty}^2(v,z) C^{mH}(v) \ol{\pi}_{v}(\diff z) \diff v \bigg| \notag \\
	=& \bigg|  \int_0^t 
	   \int_{n^{\beta}\mathcal E}  f(z) g(z) \mathcal T_{\infty}^{2}(n^{\alpha} v,z) \lambda_{n}^{mH}(v)   \pi_{n^{\alpha}v}(\diff z)  \diff v
      - \int_0^t \int_{\mathcal E} f(z) g(z)\mathcal T_{\infty}^2 (v,z) C^{mH}(v) \ol{\pi}_{v}(\diff z) \diff v \bigg|\notag \\
    \leq & \bigg|  \int_0^t 
	    \int_{n^{\beta}\mathcal E} f(z) g(z) \mathcal T_{\infty}^{2}(n^{\alpha} v,z) \left( \lambda_{n}^{mH}(v)  - C^{mH}(v) \right) \pi_{n^{\alpha}v}(\diff z)  \diff v \bigg| 
    \notag \\
    &~~~~~~ + \bigg|  \int_0^t 
	    \int_{n^{\beta}\mathcal E}  f(z) g(z) \left( \mathcal T_{\infty}^{2}(n^{\alpha} v,z) - \mathcal {T}_{\infty}^{2}( v,z)\right)    C^{mH}(v)   \pi_{n^{\alpha}v}(\diff z)  \diff v  \bigg| \notag \\
    &~~~~~~ + \bigg|  \int_0^t 
	    \int_{n^{\beta}\mathcal E}  f(z) g(z) \mathcal {T}_{\infty}^{2} ( v,z) C^{mH}(v) \left( \pi_{n^{\alpha}v}(\diff z) - \ol{\pi}_{v}(\diff z) \right)  \diff v  \bigg| \notag \\
    \leq & ||f||_{\mathcal{E}} ||g||_{\mathcal{E}}
        \bigg( \bigg| \int_0^t \int_{n^{\beta}\mathcal E} 
        \mathcal T_{\infty}^{2}(n^{\alpha} v,z) \left( \lambda_{n}^{mH}(v)  - C^{mH}(v) \right) \pi_{n^{\alpha}v}(\diff z)  \diff v  \bigg| \notag \\
     &~~~~~~ + \bigg|  \int_0^t 
	    \esp_{z\in \mathcal{E} } \left| \mathcal T_{\infty}^{2}(n^{\alpha} v,z) - \mathcal {T}_{\infty}^{2}( v,z)\right|     C^{mH}(v)   \diff v  \bigg| \notag \\
     & ~~~~~~ +  \bigg|  \int_0^t 
	       \int_{\mathcal E}  \mathcal {T}_{\infty}^{2}(v,z) C^{mH}(v) \left( \pi^{n}_{v}(\diff z) - \ol{\pi}_{v}(\diff z) \right)  \diff v  \bigg| \bigg).
	\end{align}
	As the results of Assumption \ref{assm_pi_FLLN}, Theorem \ref{thm_intensity_LLN} and Proposition \ref{prop_upsilon_tilde_upsilon_L_2}, the above difference entity (\ref{inequality_upsilon_fclt}) converges to zero in probability for all $t\geq 0$ as $n\ra \infty$. Also the expected maximum squared jump of $\left\{\frac{1}{n^{\alpha/2}}\widetilde{\Upsilon}_{n,t}(f), t\geq 0\right\}$ is bounded by $\mathcal{O}(\frac{1}{n^{\alpha}})$, which eventually vanishes to zero as $n\ra \infty$. Finally, substituting $f\equiv \bm 1$ in (\ref{def_Upsilon_{n}_(t)_diffusion}), in accordance with the martingale functional central limit theorem in view of Theorem 7.1.4 of \cite{ethier2009markov}, the requisite weak convergence of $\left \{\frac{1}{n^{\alpha/2}}\widetilde{\Upsilon}_{n,t}(f), t\geq 0\right\}$ follows.
\end{proof}

\begin{proof}[Proof of Theorem \ref{thm_FCLT_intensity_process}]
	By definitions (\ref{def_E(lambda_n(t))}), (\ref{def_Upsilon_{n}_(t)_diffusion}), we write
	\begin{align} \label{inequality_fclt_intensity_1}
    &	n^{\alpha/2} \left(  \int_0^t \lambda _{n}(s)\diff s-\int_0^t C^{mH}(s)\diff s\right) \notag \\ 
	=&n^{\alpha/2} \left( \int_0^t \lambda^{mH}_{n}(s)\diff s- \int_0^t \E \lambda ^{mH}_{n}(s) \diff s\right)+ n^{\alpha/2} \left( \int_0^t \E \lambda ^{mH}_{n}(s) \diff s-\int_0^t C^{mH}(s) \diff s\right) \notag \\
	=& \widehat{\Upsilon}_{n}(t)+ n^{\alpha/2} \left(  \int_0^{t}\mu_{n}^{mH}(s)\diff s-\int_0^t \ol{\mu}^{mH}(s)\diff  s \right)\nonumber\\& 
   \hspace{0.5cm}  + n^{\alpha/2} \left( \int_0^{t} \int_0^{s} \ol{\Gamma}_{n}(s,v)\mu_{n}^{mH}(v)\diff v \diff s -  \int_0^t \frac{\Psi(s)}{1-\Psi(s)} \mu^{mH}(s) \diff s	 \right) \notag \\
    =& \widehat{\Upsilon}_{n}(t)+ n^{\alpha/2} \left(  \int_0^{t}\mu_{n}^{mH}(s)\diff s-\int_0^t \ol{\mu}^{mH}(s)\diff s \right)+ n^{\alpha/2}  \int_0^t \left( \int_{v}^{t} \ol{\Gamma}_{n}(s,v) \diff s - \frac{\Psi(v}{1-\Psi(v)} \right) \ol{\mu}^{mH}(v)   \diff v.
	\end{align}
	Now, by view of Proposition \ref{prop_fclt_upsilon_{n}(t)}, the martingale part $\{\widehat{\Upsilon}_{n}(t), t\geq 0\}$ of (\ref{inequality_fclt_intensity_1}) weakly converges to the Gaussian process $\{\mathcal B(t), t\geq 0\}$ in $\md \left([0,\infty), \mathbb{R} \right)$. Finally, Assumptions \ref{assm_base_intensity_FCLT} and \ref{assm_Gamma_FCLT} and the application of the continuous mapping theorem conclude the requisite convergence.
 \end{proof}

\begin{proof}[Proof of Theorem \ref{thm_FCLT_N^{mH}(t)}]
     Firstly, for given $\mathscr{B}_{\mathcal E}$-measurable functions $f, g \in \mathbb{C}_{b}( \mathcal{E}, \mathbb{R})$, we decompose 
	\begin{align}\label{proof_N_mH_FCLT_1}
	\widehat{N}^{mH}_{n,t}(f)=
	&n^{1-\delta} \left( \frac{1}{n^{\alpha}} N^{mH}_{n,t}(f)- \overline{N}^{mH}(t)(f)\right) \notag \\
	&= n^{1-\delta -\alpha} \int_0^t \int_{\mathcal E} f(z) M_{n} ^{mH}(\diff s , \diff z) \notag \\
	  & ~~+ n^{1-\delta}   \int_0^t \int_{\mathcal E} f(z)
    \left( \lambda^{mH}_{n}(s) - C^{mH}(s) \right) \pi^{n}_{s} (\diff z)  \diff s\nonumber \\&
      \hspace{0.5cm}+ n^{1-\delta} \int_0^t \int_{\mathcal{E}} f(z) C^{mH}(s) \left( \pi^{n}_{s}(\diff z)- \ol{\pi}_{s}(\diff z ) \right)  \diff s  .
	\end{align}
	Let us define
	\begin{align*}
	\widehat{N}^{mH}_{n,1,t}(f) = &  n^{1-\delta -\alpha} \int_0^t \int_{\mathcal E} f(z) M_{n} ^{mH}(\diff s , \diff z) ,\\
	\widehat{N}^{mH}_{n,2,t}(f) =& n^{1-\delta}   \int_0^t \int_{\mathcal E} f(z)
    \left( \lambda^{mH}_{n}(s) - C^{mH}(s) \right) \pi^{n}_{s} (\diff z)  \diff s, \\
    \widehat{N}^{mH}_{n,3,t}(f) =& n^{1-\delta} \int_0^t \int_{\mathcal{E}} f(z) C^{mH}(s) \left( \pi^{n}_{s}(\diff z)- \ol{\pi}_{s}(\diff z ) \right)  \diff s, 
	\end{align*}
    such that $\widehat{N}^{mH}_{n,t}(f)= \widehat{N}^{mH}_{n,1,t}(f)+\widehat{N}^{mH}_{n,2,t}(f)+\widehat{N}^{mH}_{n,3,t}(f)$.

	Under the given hypotheses, $\{n^{1-\delta -\alpha} \int_0^t \int_{\mathcal E} f(z) M_{n} ^{mH}(\diff s , \diff z) , t\geq 0\}$ is an $\mathscr{F}_{n}$-square integrable martingale with quadratic variation process 
	\begin{align*}
	&\left \langle n^{1-\delta -\alpha} \int_0^. \int_{\mathcal E} f(z) M_{n} ^{mH}(\diff s , \diff z) , n^{1-\delta -\alpha} \int_0^. \int_{\mathcal E} g(z) M_{n} ^{mH}(\diff s , \diff z) \right \rangle (t) \\
	&= n^{2-2\delta-\alpha} \int_0^t \int_{\mathcal E} f(z) g(z) \lambda_{n}(s) \pi^{n}_{s}( \diff z) \diff s, 
	\end{align*}
    which converges to $\bm{1}(\delta=1-\alpha/2) \int_0^t \int_{\mathcal E} f(z)g(z) C^{mH}(s) \ol{\pi}_s ({\diff z}) \diff s$ in probability as $n\ra \infty$ for every $t\geq 0$. In order to the verify that the expected maximum square jump of the martingale $\{\widehat{N}^{mH}_{n,1,t}(f), t\geq 0 \}$ vanishes to zero as $n\ra \infty$, i.e., 
    \begin{equation*}
        \lim_{n\ra \infty} \E \left( \sup_{0\leq t \leq T} |\widehat{N}^{mH}_{n,1,t}(f) - \widehat{N}^{mH}_{n,1,t-}(f) |^2 \right) =0.
    \end{equation*}
    By the cumulative jumps of the Poisson random measure $N_{n}^{mH}(\diff s, \diff z)$, 
    \begin{align*}
        \E \left( \sup_{0\leq t \leq T} |\widehat{N}^{mH}_{n,1,t}(f) - \widehat{N}^{mH}_{n,1,t-}(f) |^2 \right)
        \leq & \E \left( \sum_{0\leq t \leq T}|\widehat{N}^{mH}_{n,1,t}(f) - \widehat{N}^{mH}_{n,1,t-}(f) |^2 \right)\\
        = & n^{2-2\delta-2\alpha} \E \left(
        \int_0^T \int_{\mathcal{E} } |f(z)|^{2} N^{mH}_{n}( \diff s, \diff z) \right)\\
        & \leq n^{2-2\delta-2\alpha} ||f||_{\mathcal {E}}^2,
    \end{align*}
    and under the diffusion scaling $\delta=1-\alpha/2$, the above cumulative jumps is bounded by $\mathcal{O}(1/n^{\alpha})$, which eventually vanishes to zero as $n\ra \infty$.
    Therefore, by martingale central limit theorem $\{\widehat{N}^{mH}_{n,1,t}(f), t\geq 0\}$ converges in distribution to the Gaussian process $\{\mathcal W_{t}(f), t\geq 0\}$ with $\E (\mathcal W_{t}(f))=0$ and 
	\begin{equation*}
	\E \left( \mathcal W_{t}(f), \mathcal W_{t}(g) \right)
	=\int_{0}^{t} \int_{\mathcal E} f(z) g(z) C^{mH}(s) \ol{\pi}_{s}(\diff z) \diff s.
	\end{equation*}

    Next, in view of (\ref{def_Z_mH_diffusion}), we can rewrite $\widehat{N}^{mH}_{n,2,t}(f)$ as 
    \begin{equation*}
    \widehat{N}^{mH}_{n,2,t}(f)= n^{1-\delta-\alpha/2}   \int_0^t F_{f}^{n}(s)   \diff \widehat{Z}_{n}^{mH}(s), 
    \end{equation*}
    where 
    \begin{equation*}
     F_{n}^{f}(t):=\int_{\mathcal E} f(z)\pi^{n}_{t}(\diff z) .
    \end{equation*}
     $\{\widehat{Z}^{mH}_{n}(t), t\geq 0\}$ is absolute continuous with respect to the Lebesgue measure a.e., and converges in distribution to the diffusion process $\{\widehat{\mathcal{B}}^{mH}(t), t\geq 0\}$ in $\mathbb{C}\left([0,\infty),\mathbb{R}\right)$ as $n\ra \infty$ by virtue of Theorem \ref{thm_FCLT_intensity_process}. Consequently, it follows from Theorem 11.4.5 of \cite{whitt2002stochastic} and Assumption \ref{assm_pi_FLLN} that $\{(F_{n}^{f}(t), \widehat{Z}^{mH}_{n}(t)), t\geq 0\}$ jointly converges to $\{(F^{f}(t), \widehat{\mathcal{B}}^{mH}(t)), t\geq 0\}$
     in $\mathbb{D}\left([0,\infty), \mathbb{R}\right) \times \mathbb{C}\left([0,\infty), \mathbb{R}\right)$ embedded with Skorokhod product topology. However, the aforementioned joint convergence is not sufficient to conclude the Skorokhod weak convergence of $\{\widehat{N}^{mH}_{n,2,t}(f), t\geq 0\}$ in $\mathbb{D}\left([0,\infty), \mathbb{R}\right)$. For this purpose, we require the sequence of semimartingales $\{\{\widehat{Z}^{mH}_{n}(t), t\geq 0\}, n\in \mathbb{N}\}$ to be \textit{predictable uniformly tight (P-UT)} (see \cite{jacod2013limit}, \cite{sen2026multiscale}), which follows from Theorem VI.6.13 of \cite{jacod2013limit} as the predictable entity of the associated martingale $\{\widehat{\Upsilon}^{n}(t), t\geq 0\}$ is $\mathbb{C}$-tight by (\ref{QV_upsilon_n_t(f,g)}). In the end, Theorem VI.6.22 of \cite{jacod2013limit} concludes the convergence of $\{\widehat{N}^{mH}_{n,2,t}(f), t\geq 0\}$ to the integral process $\bm{1}(\delta= 1-\alpha/2)\{\int_0^t \int_{\mathcal{E}} f(z) \ol{\pi}_{s}(\diff z) \diff \widehat{\mathcal{B}}^{mH}(s), t\geq 0\}$ in $\mathbb{D}\left([0,\infty), \mathbb{R}\right)$ as $n\ra \infty$.

     Again, by Assumption \ref{assm_pi_FCLT}, it follows that the deterministic process $\{\widehat{N}^{mH}_{n,3,t}(f), t\geq 0\}$ converges to $\{\int_0^t \int_{\mathcal{E}}f(z) C^{mH}(s) \widehat{\pi}_{s}(\diff z) \diff s), t\geq 0\}$ as $n \ra \infty$.

     In the end, by virtue of Theorem 11.4.4 and 11.4.5 in \cite{whitt2002stochastic}, $\{(\widehat{N}^{mH}_{n,1,t}(f),\widehat{N}^{mH}_{n,2,t}(f)\widehat{N}^{mH}_{n,3,t}(f)), t\geq 0\}$ jointly converges to $\{(\bm{1}(\delta= 1-\alpha/2)\mathcal W_{t}(f), \bm{1}(\delta= 1-\alpha/2) \int_0^t \int_{\mathcal{E}} f(z) \ol{\pi}_{s}(\diff z) \diff \widehat{\mathcal{B}}^{mH}(s), 
     \int_0^t \int_{\mathcal{E}}f(z) C^{mH}(s) \widehat{\pi}_{s}(\diff z) \diff s), t\geq 0\}$ in $\mathbb{D}\left([0,\infty), \mathbb{R}\right)^3$ as $n\ra \infty$. Since the limiting processes have continuous trajectories $\mathbb{P}$-a.s, the requisite convergence of (\ref{proof_N_mH_FCLT_1}) is the direct consequence of the continuous mapping theorem.

    Finally after substituting $f\equiv \bm{1}$, we get the desired weak convergence. 	
     The proof is complete.
	\end{proof}

   \begin{proof}[Proof of Theorem \ref{thm_FCLT_Y_n(t)_process}]
   In the beginning of the proof, we decompose the process $\{\widehat{Y}^{mH}_{n}(t), t\geq 0\}$ as follows:
  \begin{align}\label{proof_Y_mH_FCLT_1}
  \widehat{Y}^{mH}_{n}(t)=
 \widehat{Y}^{mH}_{n,1}(t)+\widehat{Y}^{mH}_{n,2}(t)+\widehat{Y}^{mH}_{n,3}(t)+\widehat{Y}^{mH}_{n,4}(t),
  \end{align}
  where 
  \begin{align*}
   \widehat{Y}^{mH}_{n,1}(t)= &n^{1-\delta-\alpha}\int_0^t \int_{\mathcal{E}}\theta_{n}(t-s) M^{mH}_{n}(\diff s, \diff z),\\
   \widehat{Y}^{mH}_{n,2}(t)= &n^{1-\delta} \int_0^t \int_{\mathcal{E}} \theta_{n}(t-s) \left( \lambda^{mH}_{n}(s) - C^{mH}(s) \right) \pi_{s}^{n}(\diff z) \diff s, \\
   \widehat{Y}^{mH}_{n,3}(t)= &n^{1-\delta} \int_0^t \int_{\mathcal{E}} \left(\theta_{n}(t-s,z) - \ol{\theta}(t-s,z) \right) C^{mH}(s) \pi_{s}^{n}(\diff z) \diff s,\\
   \widehat{Y}^{mH}_{n,4}(t)= &
    n^{1-\delta} \int_0^t \int_{\mathcal{E}} \ol{\theta}(t-s,z) C^{mH}(s) \left(\pi^{n}_{s}(\diff z) - \ol{\pi}_{s}(\diff z) \right)  \diff s.
   \end{align*}
  Recall that $\{\widehat{Y}^{mH}_{n,1}(t), t\geq 0\}$ is an $\mathscr{F}_{n}$-square integrable martingale with quadratic variation process 
  \begin{equation*}
    \left \langle \widehat{Y}^{mH}_{n,1}\right \rangle (t)
    = n^{2-2\delta -\alpha} \int_0^t \int_{\mathcal{E}}
      \theta^{2}_{n}(t-s) \lambda^{mH}_{n}(s)\pi_{s}^{n}(\diff z) \diff s.
  \end{equation*}
  Applying the similar arguments used in the proof of Theorem \ref{thm_LLN_Y^{mH}(t)}, it can be easily verified that for any $T>0$ and $\epsilon>0$,
  \begin{align}
   & \lim_{n\ra \infty} \prob \left( \sup_{0\leq t \leq T} \left| \int_0^t \int_{\mathcal E}  \theta^{2}_{n}(t-s) \lambda^{mH}_{n}(s)\pi_{s}^{n}(\diff z) \diff s
   - \int_0^t \int_{\mathcal E} \ol{\theta}^2(t-s,z) C^{mH}(s) \ol{\pi}_s(\diff z) \diff s	\right|>\epsilon \right)=0 .
   \end{align}
  Also, analogous to the above proof, the expected maximum square jump of $\{\widehat{Y}^{mH}_{n,1}(t), t\geq 0\}$ is bounded by $\mathcal{O}(n^{2-2\delta-2\alpha})$, and consequently $\bm{1}(\delta=1-\alpha/2) \mathcal{O}(1/n^{\alpha})$ which converges to zero as $n\ra \infty$. 
  Therefore, by martingale central limit theorem $\{\widehat{Y}^{mH}_{n,1}(t), t\geq 0\}$ converges in distribution to the Gaussian process $\{\bm{1}(\delta=1-\alpha/2)\mathcal{W}_{\ol \theta}(t), t\geq 0 \}$ in $\md\left([0,\infty), \mathbb{R}\right) $ as $n\ra \infty$, where the limiting process $\{\mathcal{W}_{\theta}(t), t\geq 0 \}$ is characterized by 
  \begin{equation}\label{QV_W_theta_Y_mH_FCLT}
    \E \left( \mathcal{W}_{\ol \theta}(t) \right) =0 , \quad 
    \E \left(\mathcal{W}_{\ol \theta}(t)\mathcal{W}_{\ol \theta}(s) \right)= \int_0^{t\wedge s} \int_{\mathcal E} \ol{\theta}^2(t-v,z) C^{mH}(v) \ol{\pi}_v(\diff z) \diff v.
  \end{equation}

Analogous to the proof of Theorem \ref{thm_FCLT_N^{mH}(t)}, we rewrite $\widehat{Y}^{mH}_{n,2}(t)$ as 
\begin{align*}
    \widehat{Y}^{mH}_{n,2}(t)= n^{1-\delta-\alpha/2}
    \int_0^t \int_{\mathcal{E}}\theta_{n}(t-s,z) \pi^{n}_{s}(\diff z) \diff \widehat{Z}^{mH}_{n}(s).
\end{align*}
By Assumptions \ref{assm_pi_FLLN} and \ref{assm_theta_FLLN}, it follows that the deterministic entity $\int_{\mathcal{E}}\theta_{n}(t-s,z) \pi^{n}_{s}(\diff z)$ converges to $\int_{\mathcal{E}}\ol{\theta}(t-s,z) \ol{\pi}_{s}(\diff z)$ uniformly on compact sets as $n\ra \infty$.
Applying the similar P-UT arguments as used in the above proof, it can be proved that $\{\widehat{Y}^{mH}_{n,2}(t), t\geq 0\}$ weakly converges to $\{\bm{1}(\delta=1-\alpha/2)\int_0^t \int_{\mathcal{E}} \ol{\theta}(t-s,z) \ol{\pi}_{s}(\diff z) \diff \widehat{\mathcal{B}}^{mH}(s), t\geq 0\}$ in $\mathbb{D}\left( [0, \infty), \mathbb{R}\right)$ as $n\ra \infty$.

Next, it directly follows from Assumptions \ref{assm_pi_FLLN}, \ref{assm_theta_FLLN}, \ref{assm_pi_FCLT}, \ref{assm_theta_FCLT} that $\{\widehat{Y}^{mH}_{n,3}(t), t\geq 0\}$ and $\{\widehat{Y}^{mH}_{n,4}(t), t\geq 0\}$ respectively converges to $\{\int_0^t \int_{\mathcal{E}} \widehat{\theta}(t-s,z) C^{mH}(s) \ol{\pi}_{s}(\diff z) \diff s, t\geq 0\}$ and $\{ \int_0^t \int_{\mathcal{E}} \ol{\theta}(t-s,z) C^{mH}(s) \widehat{\pi}_{s}(\diff z) \diff s, t\geq 0\}$.

In the end, by virtue of Theorems 11.4.4 and 11.4.5 of \cite{whitt2002stochastic}, $\{(\widehat{Y}^{mH}_{n,1}(t), \widehat{Y}^{mH}_{n,2}(t), \widehat{Y}^{mH}_{n,3}(t), \widehat{Y}^{mH}_{n,4}(t)),t\geq 0\}$
 weakly converges to $\{(\bm{1}(\delta=1-\alpha/2) \mathcal{W}_{\ol \theta}(t), 
\bm{1}(\delta=1-\alpha/2) \int_0^t \int_{\mathcal{E}} \ol{\theta}(t-s,z) \ol{\pi}_{s}(\diff z) \diff \widehat{\mathcal{B}}^{mH}(s), 
\int_0^t \int_{\mathcal{E}} \widehat{\theta}(t-s,z) C^{mH}(s) \ol{\pi}_{s}(\diff z) \diff s,
\int_0^t \int_{\mathcal{E}} \ol{\theta}(t-s,z) C^{mH}(s) \widehat{\pi}_{s}(\diff z) \diff s) ,t \geq 0\} $
with trajectories in $\md \left([0, \infty), \mathbb{R}\right)^{4}$. Finally, the continuous mapping theorem results the desired convergence of (\ref{proof_Y_mH_FCLT_1}).

\end{proof}

\section{Conclusion and Future Work}\label{section_conclusion}
In this paper, we establish the functional limit theorems including the functional laws of large numbers and functional central limit theorems for the marked Hawkes process along with the shot noise process under the multi-scaling high intensity regime. This particular asymptotic regime is generated by multiplying the time parameter by $n^{\alpha}$ for some $\alpha>0$. Depending on $\alpha$, the diffusion scale parameter $\delta $ takes the value accordingly to derive the functional limit theorems under suitable assumptions, which develops a new asymptotic framework to study significantly different from the conventional asymptotic regime. Inside the new proposed asymptotic framework, we prove the diffusion approximations for the cumulative intensity process associated with the marked Hawkes process. Later, we obtain that the limiting  processes of the marked Hawkes process as well as shot noise process resulting from the functional central limit theorems is a sum of a Gaussian noise and diffusion process (associated with an another Gaussian process).

This work focuses on the univariate Hawkes model, and extending the analysis to the multivariate setting would provide a broader understanding of interactions across multiple components. Investigating how our asymptotic regime and marked structure behave in the multivariate case is an important direction for future research. Earlier, the Hawkes process without incorporating marks, i.e., for $\phi(t-\tau_{i}, \eta_{i}(\tau_{i}))=\phi(t-\tau_{i})$, $i\in \mathbb{N}$, 
 has been studied by \cite{bacry2013some} under multi-dimensional framework. Another direction for future work could involve studying asymptotic phenomena, including mean-field limits and large deviation principles for the framework developed in this paper. Moreover, it would be interesting to study the limiting behaviour of the proposed processes when the marks in the self-exciting function are scaled jointly with the time parameter.

\medskip{}
\noindent
\textbf{Author Contributions} Ankita Sen: Identified the problem, designed the model and analysed. Dharmaraja Selvamuthu: Revised the model and Critical revision of the article. N. Selvaraju: Updated the final manuscript.

\medskip{}
\noindent
\textbf{Availability of Data and Materials} Data availability is not applicable to this article as no new data were created or analysed in this study.

\medskip{}
\noindent
 \textbf{Declarations} The authors declare no competing interests.

\end{document}